\documentclass[a4paper]{article}
\usepackage[T1]{fontenc}
\usepackage{amsfonts,amsmath,amsthm,amssymb}
\usepackage{calc}
\usepackage{gensymb}
\usepackage[nottoc, notlof, notlot]{tocbibind}
\usepackage{fullpage}
\usepackage{mathrsfs}  
\usepackage{bigints}
\usepackage{indentfirst}
\usepackage{enumerate}
\usepackage{authblk}
\usepackage{hyperref}
\usepackage{stmaryrd}

\numberwithin{equation}{section}
\DeclareMathOperator{\tr}{Tr}
\DeclareMathOperator{\ts}{tr}

\DeclareMathOperator{\id}{id}
\DeclareMathOperator{\ev}{ev}

\newcommand{\norm}[1]{\left\Vert #1\right\Vert}

\begin{document}

\newtheorem{theorem}{Theorem} [section]
\newtheorem{prop}[theorem]{Proposition} 
\newtheorem{defi}[theorem]{Definition} 
\newtheorem{exe}[theorem]{Example} 
\newtheorem{lemma}[theorem]{Lemma} 
\newtheorem{rem}[theorem]{Remark} 
\newtheorem{cor}[theorem]{Corollary} 
\newtheorem{conj}[theorem]{Conjecture}
\renewcommand\P{\mathbb{P}}
\newcommand\E{\mathbb{E}}
\newcommand\N{\mathbb{N}}
\newcommand\1{\mathbf{1}}
\newcommand\C{\mathbb{C}}
\newcommand\CC{\mathcal{C}}
\newcommand\M{\mathbb{M}}
\newcommand\R{\mathbb{R}}
\newcommand\U{\mathbb{U}}
\newcommand\A{\mathcal{A}}
\newcommand\B{\mathcal{B}}
\newcommand\F{\mathcal{F}}
\renewcommand\i{\mathbf{i}}
\newcommand\x{\mathbf{x}}
\newcommand\y{\mathbf{y}}
\renewcommand\S{\mathcal{S}}
\renewcommand\d{\partial_i}
\renewcommand\.{\ .}
\renewcommand\,{\ ,}

\def\etc{,\dots ,}

\allowdisplaybreaks

 \begin{minipage}{0.85\textwidth}
 	\vspace{2.5cm}
 \end{minipage}
 \begin{center}
 	\large\bf Asymptotic freeness through unitaries generated by polynomials of \\ Wigner matrices
 	
 \end{center}

 \renewcommand{\thefootnote}{\fnsymbol{footnote}}	
 \vspace{0.8cm}
 
 \begin{center}
 	\begin{minipage}{1.5\textwidth}
 		
 		\begin{minipage}{0.3\textwidth}
 			\begin{center}
 				F\'elix Parraud\\
 				\footnotesize 
 				{KTH Royal Institute of Technology}\\
 				{\it parraud@kth.se}
 			\end{center}
 		\end{minipage}
 		\begin{minipage}{0.3\textwidth}
 			\begin{center}
 				Kevin Schnelli\\
 				\footnotesize 
 				{KTH Royal Institute of Technology}\\
 				{\it schnelli@kth.se}
 			\end{center}
 		\end{minipage}
 	\end{minipage}
 \end{center}
 
 \bigskip
\bigskip

\begin{abstract}
		We study products of functions evaluated at self-adjoint polynomials in deterministic matrices and independent Wigner matrices; we compute the deterministic approximations of such products and control the fluctuations. We focus on minimizing the assumption of smoothness on those functions while optimizing the error term with respect to $N$, the size of the matrices. As an application, we build on the idea that the long-time Heisenberg evolution associated to Wigner matrices generates asymptotic freeness as first shown in \cite{thermalidsl}. More precisely  given~$P$ a self-adjoint non-commutative polynomial and $Y^N$ a $d$-tuple of independent Wigner matrices, we prove that the quantum evolution associated to the operator $P(Y^N)$ yields asymptotic freeness for large times.
\end{abstract}

 \vspace{5mm}
 
\footnotetext{\\ \indent \textit{Date}: March 14, 2024. 

\indent F.P. and K.S. are supported by the Swedish Research Council and the Knut and Alice Wallenberg Foundation.}

\section{Introduction and main results}

In this paper, we consider polynomials in several independent Wigner matrices $Y_1^N,\dots,Y^N_d$. In the early nineties, in his seminal work \cite{DVoi}, Voiculescu showed that the fitting framework to understand spectral properties of such polynomials is Free Probability, a non-commutative probability theory with a notion of freeness analogous to independence in classical probability. For Gaussian Wigner matrices he proved that for any given collection of polynomials $P_1,\dots,P_k$ and $i_j\in [1,d]$,  almost surely,
\begin{equation}
	\label{oirdkr}
	\lim_{N\to\infty} \ts_{N}\big( P_1(Y_{i_1}^N)\cdots P_k(Y_{i_k}^N)\big) = \tau\big(P_1(x_{i_1})\cdots P_k(x_{i_k})\big),
\end{equation}
where $\ts_N$ denotes the normalized trace on $\M_N(\C)$; $x_1,\dots,x_d \in \CC_d$ is a system of $d$ free semicircular variables and $\tau$ is the trace on the $\CC^*$-algebra $\CC_d$; see Definition \ref{3tra}. This result was generalized in numerous ways, to begin with in \cite{wigpague}, the author extended this result to non-Gaussian Wigner matrices coupled with deterministic matrices under certain assumptions; see also Theorem 5.4.5 of \cite{alice} for a proof without those assumptions on the deterministic matrices. It is further possible to consider continuous functions but only by approximating them with polynomials. For such results, one usually first proves the convergence in expectation, and then uses concentration inequalities to establish the almost sure convergence. One of the limitations of such combinatorial methods is that they are not well-suited to obtain estimates on the convergence rate in \eqref{oirdkr}. Strong quantitative estimates for smooth functions were first obtained in \cite{hT} by Haagerup and Thorbj\o rnsen in the Gaussian case. They studied not only the convergence of the trace but also of the operator norm.  These results were extended to include deterministic matrices in \cite{male,topoexp}, and to more general Wigner matrices in \cite{anderste,capbelin}. Beyond the question of polynomials, it is also worth noting that the case of non-commutative rational functions was tackled in \cite{ncfu,yinsldc}. Those papers however focus on studying the expectation rather than proving almost sure estimates, i.e.\ estimates on the difference between our random variable and its deterministic limit which hold with high probability. While measure concentration inequalities are usually sufficient to deduce almost sure results, one does not necessarily have such tools for general Wigner matrices. Indeed if the law of the entries of our random matrices satisfy a log-Sobolev inequality then so do their joint laws; see Section 2.3 of \cite{alice} for a good introduction on the topic. However, in this paper we do not make this kind of assumption on our random matrices. A possible way to address this issue would be the approach of \cite{ToddDavid} which uses mollified log-Sobolev inequalities to prove concentration estimates. In this paper we choose a more direct approach by studying high order moments. 

The main results of our paper are a high probability estimates for the trace as well as the matrix entries of any products of sufficiently smooth functions evaluated at polynomials in deterministic and independent Wigner matrices. We are in particular interested in optimizing the error term with respect to not only $N$ but also the derivatives of our functions. This will be especially useful for Theorem \ref{dliuf} later below. 

The following notions will be used in our statement of Theorem~\ref{main_result}. A square random matrix $Y=(Y)_{i,j}$ of size $N$ is a Wigner matrix if it is Hermitian or real symmetric and its entries are independent up to the symmetry constraints. The entries are assumed to be centered with variances $\E[|Y_{i,j}|^2]=\frac{1}{N}$, for all $i\not=j$, and they satisfy the moment bounds
\begin{align}
\sup_{N\in\N, 1\leq i,j\leq N} \E\left[\left| \sqrt{N} Y_{i,j}\right|^p\right] < \infty,
\end{align}
for any $p\ge 2$, c.f. Definition \ref{defwigner} below. Given a sequence of random variables $(\mathcal{X}_N,\mathcal{Y}_N)_{N\geq 1}$ as well as a sequence $(\varepsilon_N)_{N\geq 1}$ of non-negative real numbers, we say that with high probability
$$ \mathcal{X}_N = \mathcal{Y}_N + \mathcal{O}(\varepsilon_N),$$
if for any $k>0$, there exists a numerical constant $C$ such that, $\P(|\mathcal{X}_N-\mathcal{Y}_N|\geq C\ \varepsilon_N)\leq N^{-k}$ holds for any $N$ sufficiently large, see Definition~\ref{definition of high probability}.

\begin{theorem}\label{main_result}
	Let the following objects be given,
	\begin{itemize}
		\item $Y^N= (Y_1^N,\dots,Y_d^N)$ independent real symmetric or complex Hermitian Wigner matrices of size $N$ defined as in Definition \ref{defwigner} below, 
		\item $A^N=(A_1^N,\dots,A_q^N)$ deterministic matrices of size $N$, such that $\sup_{1\leq i\leq q,N\in\N^*} \norm{A_i^N} <\infty$,
		\item $x=(x_1,\dots,x_d)$ a system of free semicircular variables, free from $A^N$, i.e. they belong to the free product  $\mathcal{C}_d(x)*\M_N(\C)$ where $\CC_d$ is the $\mathcal{C}^*$-algebra generated by $x$ (see Definition \ref{3tra} below),
		\item $f_1,\dots,f_k$, $k\geq 1$, functions such that either $f_i = \id_{\R}$ or there exists a complex-valued measure $\mu_i$ such that
		\begin{equation}
			\label{okmskfds}
			\forall t\in \R,\quad f_i(t) = \int_{\R} e^{\i ty} d\mu_i(y),
		\end{equation}
		\item $P_1,\dots,P_k$ non-commutative polynomials, such that whenever $f_i\neq \id_{\R}$, $P_i$ is self-adjoint (see Subsection \ref{3poly}).
	\end{itemize}
	Then with the convention $\norm{f_i}_4 =1 $ if  $f_i = \id_{\R}$, and otherwise, 
	\begin{equation}
		\label{lkdld}
		\norm{f_i}_4 := \int_{\R} (1 + y^4)\ d|\mu_i|(y),
	\end{equation}
    where $|\mu|$ is the variation of the measure $\mu$. Then we have the following result. For any $\varepsilon>0$, with high probability,
	\begin{align}
		\label{trace}
		&\ts_{N}\left( f_1(P_1(Y^N,A^N)) \cdots f_k(P_k(Y^N,A^N)) \right)  \nonumber \\
		&\quad\quad\quad\quad\quad= \tau_{N}\left( f_1(P_1(x,A^N)) \cdots f_k(P_k(x,A^N)) \right) + \mathcal{O}\left(N^{\varepsilon}\frac{\max_{i} \norm{f_i}_4}{N}\right),
	\end{align}
	where $\ts_{N}$ is the normalized trace on $\M_N(\C)$, while $\tau_N$ is the trace on the free product $\mathcal{C}_d*\M_N(\C)$; see Definition \ref{3tra}. Moreover for $\x ,\y\in\C^N$, we also have with high probability that
	\begin{align}
		\label{scalar}
		&\left\langle \x, f_1(P_1(Y^N,A^N)) \cdots f_k(P_k(Y^N,A^N)) \y \right\rangle \nonumber \\
		&\quad\quad\quad\quad\quad= \left\langle \x, E_{\M_N(\C)}\big[f_1(P_1(x,A^N)) \cdots f_k(P_k(x,A^N))\big]\y \right\rangle + \mathcal{O}\left(N^{\varepsilon}\frac{\max_{i} \norm{f_i}_4}{\sqrt{N}}\norm{\x}_2\norm{\y}_2\right),
	\end{align}
	where $E_{\M_N(\C)}$ is the conditional expectation from $\mathcal{C}_d*\M_N(\C)$ to $\M_N(\C)$ (see Definition~\ref{conditionalmanqu} and the remarks that follow, notably for some special cases where this conditional expectation is easy to compute).
	Finally, if for every $i$, $Y_i^N$ is a GUE or GOE random matrix, then one can replace $\norm{f_i}_4$ by $\norm{f_i}_2$ in all of the previous formulas.
\end{theorem}

Theorem \ref{main_result} calls for the following remarks:
\begin{itemize}
	\item It is important to note that unlike previous works, we do not study the Stieltjes transform to get Equations \eqref{trace} and \eqref{scalar}, instead we study the Fourier transform of the functions $f_i$. A different strategy would be to use Helffer-Sjöstrand calculus after studying a product of resolvent, i.e of terms of the form $(z_i-P_i(Y^N,A^N))^{-1}$ for some polynomials $P_i$ and complex numbers $z_i$. But then the error term will depend on $\max_{i}|\Im z_i|^{-k}$ which translates into a bound in terms of $\max_{i} \norm{f_i}_k$. However as one can see in Equations \eqref{trace} and \eqref{scalar}, this dependence on $k$ is not optimal as soon as $k$ is larger than $4$.
	\item The norm $\norm{\cdot}_4$ is the one associated to the fourth Wiener space $W_4(\R)$; we refer to \cite{nikitopoulos}, Section~3.2, for more information on the topic. In particular, if $f$ and its Fourier transform are integrable, then~$f$ satisfies Assumption \eqref{okmskfds} with the measure $d\mu(y) = \widehat{f}(y)\ dy$. Then assuming that $\norm{f}_4$ is finite implies that $f$ is four times differentiable and
	$$ \norm{f}_4 = \int_{\R} (1+y^4) |\widehat{f}(y)|\ dy = \int_{\R} |\widehat{f}(y)|+|\widehat{f^{(4)}}(y)|\ dy.$$
	Thus heuristically $\norm{f}_4$ is related to the fourth derivative of $f$. 
	\item Note that in Theorem \ref{main_result}, one can chose the functions $(f_i)_{i\in[1,k]}$ to depend on $N$, this implies that the norm in the error term will also depend on $N$. This allows us for example to consider mesoscopic test functions of the form
	$$ f_i^N: x\mapsto g_i\big( N^a(x-E) \big), $$
	where $E\in\R, a\in [0,1/2)$ and the functions $g_i$ satisfy Assumptions \ref{okmskfds} and \ref{lkdld}. However the error term will be of order $N^{4a+\varepsilon-1}\max_{i} \norm{g_i}_4$ (or $N^{2a+\varepsilon-1}\max_{i} \norm{g_i}_2$ if for every $i$, $Y_i^N$ is a GUE or GOE random matrix) instead of $N^{\varepsilon-1}\max_{i} \norm{f_i}_4$. Besides the deterministic approximation will depend on $N$ in Equations \eqref{trace} and \eqref{scalar}.
	\item The estimates in~\eqref{trace} and~\eqref{scalar} are optimal in terms of the powers in $N$ since in the case of GUE random matrices, the difference between the left hand side and the right hand side of Equation \eqref{trace} multiplied by $N$ converges in law towards a Gaussian random variable, hence it has to be of order $1$ with respect to $N$. Indeed, thanks to Theorem 3.4 of \cite{topoexp}, one can replace the trace on the right hand side by the expectation of the left hand side. Then proving such a central limit theorem is a well-known problem, see Theorem 7.6 of \cite{aliceSDy} for the case of polynomials. However, we believe that the optimal exponent of $y=1+\max |y_i|$ in~Proposition~\ref{concentrtrace} and~\ref{concentrcoeff} should be two, as it is the case for the Gaussian ensembles, see also Theorem~2.6 in~\cite{thermalidsl} for $k\le 4$. Thus in turn we expect that the estimates in~\eqref{trace} and \eqref{scalar} hold for functions in the second Wiener space; i.e. with $\norm{f}_2$ instead of $\norm{f}_4$. 
	\item Lastly, we remark that the Wigner matrices $Y^N$ can be both real symmetric or complex Hermitian, or a mix of them in~\eqref{trace} and~\eqref{scalar}. Further, we do not need the matrix entries to have the same law and the variances of the diagonal entries are only assumed to be of order $1/N$.
\end{itemize}

One can notably compare Theorem \ref{main_result} with Theorem 2.6 of \cite{thermalidsl} and Corollary 2.7 of \cite{traceless}. The biggest difference in our work is that we consider any polynomials $P_1,\dots,P_k$ whereas in \cite{thermalidsl} the authors considered monomials of degree $1$, i.e. $P_i(Y^N,A^N) = Y_{j_i}^N$ for some $j_i\in[1,d]$, or $P_i(Y^N,A^N) = A_{j_i}^N$ for some $j_i\in[1,q]$. For such products the key result to prove Theorem 2.6 in \cite{thermalidsl} is a multi-resolvent local law (Theorem~3.4 of their paper): Let $G^N(z):=(Y^N-z)^{-1}$, $z\in\C\backslash \R$, denote the resolvent or Green function of a single Wigner matrix $Y^N$. Then the Green function $G^N(z)$ is well approximated by $m_{sc}(z)I^N$ with $m_{sc}(z)$ the Stieltjes transform of the semicircle law, in averaged sense~\cite{ErdoesKnowlesYauYin,lclclc,rigidity} as well as in isotropic sense~\cite{Alex+Erdos+Knowles+Yau+Yin,KnowlesYinIsotropic}, for $|\Im z|\gg N^{-1}$, i.e.\ down to local scales slightly above the typical eigenvalue spacing. Theorem~3.4 in~\cite{thermalidsl} and Theorem 2.5 in [10] give the deterministic approximation for multi-resolvents of the form $G^N(z_1)A_1^NG^N(z_2)A_2^N\cdots A_{k-1}^NG^N(z_k) A_{k}^N$ with optimal error terms and on local scales; where $(A_i^N)$ are as in Theorem~\ref{main_result}. The Helffer-Sj\"ostrand calculus allows then to extend this local law to observables of the form $f_1(Y^N)A_1^Nf_2(Y^N)\cdots f_k(Y^N)A_k^N$, for Sobolev functions $(f_i)$, in averaged and isotropic sense, yielding the results in Theorem~2.6 of~\cite{thermalidsl}. Yet, local laws for a single Green function of a polynomial in Wigner matrices~\cite{AndersonAnti,lcp,nemish2,nemish1} are difficult to derive, partly due to hard to check stability conditions stemming from the linearization trick~\cite{hT}. In particular for our desired results the local laws needed to be established up to the spectral edges, which for general polynomials seems challenging, not to speak of multi-resolvent local laws.

We circumvent this difficulty by working with the Fourier transform instead of the Stieltjes transform. A main novelty are concentration estimates for quantities of the form $e^{\i P_1 y_1}R_1\cdots e^{\i P_k y_k}R_k$ with $(P_i)$ and $(R_i)$ polynomials in $Y^N$ and $A^N$, and $(y_i)\in\R^k$; see Propositions~\ref{concentrtrace} and~\ref{concentrcoeff} below. These concentration results are obtained by a long-time continuous interpolation between Wigner matrices and the GUE applied to the non-commutative setting using recursive moment estimates~\cite{moment,LeeSchnellisparse}; see~\cite{KnowlesYinInterpolation,deformed} for long-time interpolations for Green functions of single Wigner matrices. This strategy does not rely on the linearization of non-commutative polynomials and hence avoids stability issues. The concentration estimates in~Propositions~\ref{concentrtrace} and~\ref{concentrcoeff} are then combined with results in~\cite{topoexp} to replace GUE matrices with semicircular variables by interpolating them with the help of free stochastic calculus. This strategy of interpolating between random matrices and free variables was first developed in \cite{un} before being refined to get a better estimate of the remainder term in \cite{topoexp}.  Via the Fourier transform our results in Theorem~\ref{main_result} follow for functions in the fourth Wiener space. 

 Another difference to~\cite{thermalidsl} is that our method focuses on algebraic rather than combinatorial aspects. In~\cite{thermalidsl} the main results, including the multi-resolvent local law, are stated using advanced combinatorics, such as free cumulant functions, partial traces or the Kreweras complement. A key insight of~\cite{thermalidsl} was to connect resolvent expansions to non-crossing partitions and free cumulants. In this paper, we rely instead on the algebraic formalism of free probability, e.g.\ free products of $C^*$-algebras or conditional expectations, that allows us to formulate the convergence results in~\eqref{trace} and~\eqref{scalar} in all generality, i.e.\ with $P_i$ arbitrary polynomials in $Y^N$ and $A^N$. In Subsection~\ref{3deffree}, we outline how to pass from an algebraic to a combinatorial formulation in some special cases.

Motivations for Theorem \ref{main_result} come from free probability and mathematical physics. As explained in~\cite{thermalidsl}, the RAGE theorem (see \cite{rageop}, Theorem 5.8) states that given $H$ a self-adjoint operator on an infinite dimensional Hilbert space, $\phi$ a vector state in the continuous spectral subspace of $H$, and $C$ a compact operator, then the Heisenberg time evolution of $C$ vanishes on $\phi$ when $t$ goes to infinity, more precisely, the Ces\`aro mean of $\langle \phi, e^{\i t H} C e^{-\i t H} \phi \rangle$ vanishes for large $t$. Thus heuristically the RAGE theorem describes the asymptotic behavior of the Heisenberg time evolution under some assumptions on $C$ and $\phi$. For example one may consider the operator $H=P(x_1,\dots,x_d)$ where $P$ is a polynomial and $(x_i)$ are free semicircular variables. It is well-known that the spectrum of a polynomial in independent Wigner matrices behaves similarly to the one of the same polynomial evaluated in free semicircular variables, see for example Theorem 5.4.2 in \cite{alice}. Consequently one expects that up to an error which depends on $N$, the Heisenberg time evolution associated with the operator $H=P(Y^N_1,\dots,Y^N_d)$ will behaves similarly. Indeed, from Theorem \ref{main_result}, we get that if $P$ is such a polynomial, then for $t$ and $N$ large enough, we have for any bounded sequence of deterministic matrices $(B^N)$ and $(C^N)$, as well as vectors $\x,\y\in\C^N$, that
\begin{equation}\label{lila1}
	\left\langle \x, e^{\i t P(Y^N)} C^N e^{-\i t P(Y^N)} \y \right\rangle \simeq  \ts_N\left( C^N\right) \times \langle \x,\y\rangle,
\end{equation}
\begin{equation}
	\label{frefewvo}
	\ts_N\left( e^{\i t P(Y^N)} C^N e^{-\i t P(Y^N)} B^N \right) \simeq  \ts_N\left( C^N\right) \times \ts_N\left(B^N \right) .
\end{equation}
That the quantities on the right sides do not necessarily converge towards $0$ comes from the fact that~$(C^N)$ may not behave asymptotically as a compact operator, indeed one could for example take~$(C^N)$ the identity matrices. Equations~\eqref{lila1} and~\eqref{frefewvo} indicate decays of correlations under the long-time Heisenberg evolution, a phenomenon referred to as thermalization in \cite{thermalidsl}. More generally, the authors showed in Corollary~2.12 of \cite{thermalidsl}, that the long-time Heisenberg evolution associated to a Wigner matrix generates asymptotic freeness. 

The first link between Free Probability and Random Matrix Theory was made by Voiculescu in the nineties. In \cite{DVoi} he showed that the trace of any polynomial in GUE random matrices converges towards the trace of the same polynomial but evaluated in free semicircular operators. He introduced the notion of asymptotic freeness accordingly: A sequence of families of random matrices $(X_1^N,\dots,X_d^N)$ is asymptotically free if for every collection of polynomials $P_1,\dots,P_n\in\C[X]$ such that $\ts_N(P_j(X_{i_j}^N))$ converges towards $0$, then if $i_j\neq i_{j+1}$ for every $j$, we have that
$$\lim_{N\to\infty} \ts_N\left( P_1(X_{i_1}^N) \cdots P_n(X_{i_n}^N) \right) =0.$$
In particular, if $(A^N)$ and $(B^N)$ are asymptotically free, then
$$\lim_{N\to\infty} \ts_N\Big( \left(A^N-\ts_N(A^N)\right) \left(B^N-\ts_N(B^N)\right)  \Big) =0.$$
Consequently, Equation \eqref{frefewvo} is a corollary of the asymptotic freeness of $e^{\i t P(Y^N)} C^N e^{-\i t P(Y^N)}$ and $B^N$. 

Given two sequences of real numbers $(u_N)$ and $(v_N)$, we write $u_N\ll v_N$ if $u_N/v_N$ converges towards~$0$. 

\begin{theorem}
	\label{dliuf}
	Let the collection of deterministic matrices $A^N=(A_1^N,\dots,A_k^N)$  and of Wigner matrices $Y^N= (Y_1^N,\dots,Y_d^N)$ be given as in Theorem~\ref{main_result}. Let $P$ be a non-constant self-adjoint non-commutative polynomial in $d$ variables. Besides we assume that for every $i$, $A_i^N$ converges in distribution towards a non-commutative random variable $a_i$; see Definition~\ref{3freeprob}. Further, let  $y_1^N\leq\dots\leq y_k^N\in\R$ be such that for any $i\in [1,k-1]$, $1 \ll y_{i+1}^N-y_i^N \ll N^{1/4}$. Then with 
	$$a_i^N:= e^{\i y_i^N P(Y^N)} A_i^N e^{-\i y_i^N P(Y^N)},$$ 
	almost surely the family of non-commutative random variables $(a_1^N,\dots,a_k^N)$ converges jointly in distribution towards $(a_1,\dots,a_k)$ where the $(a_i)$ are free. Moreover, if for every $i$, $Y_i^N$ is a GUE or GOE random matrix, then one can replace $N^{1/4}$ by $N^{1/2}$ in the previous assumptions, i.e. one can assume that for any $i\in [1,k-1]$, $1 \ll y_{i+1}^N-y_i^N \ll N^{1/2}$.
\end{theorem}

One can notably compare this theorem with Corollary 2.12 of \cite{thermalidsl} which studied the case of a single Wigner matrix, i.e. $d=1$ and $P=X_1$ in the theorem above. Besides, the equivalents of Equations \eqref{trace} and \eqref{scalar} are given in Theorem 2.6 of \cite{thermalidsl} (respectively Corollary 2.7 of \cite{traceless}), there the error term depends on the $k$-th Sobolev norm (respectively $k/2$) of the functions $(f_i)$ where $k$ is the number of functions considered. In the case of the function $f_y:t\in\R\mapsto e^{\i y t}$, the k-th Sobolev norm of $f_y$ is of order $y^k$ whereas if we use Theorem \ref{main_result}, we will get that $\norm{f_y}_4$ is of order~$y^4$. While the former is better for smaller~$k$, this dependence on $k$ requires one to assume in Corollary 2.12 of \cite{thermalidsl} that $1 \ll y_{i+1}^N-y_i^N \ll N^{1/k}$ for any~$k$, which can be improved to $N^{2/k}$ thanks to Remark 2.8 of \cite{traceless}. Consequently, in order to derive the simultaneous convergence of every moment -- i.e.\ the convergence in distribution as defined in Definition \ref{3freeprob} -- out of this result, one has to assume that the time differences go to infinity slower than any power of $N$.

One can also wonder about the thermalization decay rates in Theorem \ref{dliuf}. Indeed, given a polynomial~$P$ as defined in this theorem, the convergence rates depend on the behavior of the Fourier transform of the limiting spectral measure of~$P$ evaluated in free semicircular variables. For example, in the case of a single Wigner matrix, the spectral measure at the limit is the semicircle law which has square root decay at the spectral edges, hence this yields the exponent~$\delta_P=3/2$ in the equation below; as in Corollary 2.10 of~\cite{thermalidsl}. For general $P$, thanks to Theorem~1.1(5-6) of~\cite{density}, one can show that there exists a constant~$\delta_P>0$, such that for any polynomial~$Q$, one has that for any~$\varepsilon>0$, with high probability,
\begin{align*}
	\tau_N\left( Q(a_1^N,\dots,a_k^N)\right) = \tau\big( Q(a_1,\dots,a_k) \big) + \mathcal{O}\bigg( \frac{N^\varepsilon}{N}\max\limits_{i\neq j} | y_i^N-y_j^N |^4 + \frac{1}{\min\limits_{i\neq j} | y_i^N-y_j^N |^{\delta_P}}\bigg). 
\end{align*}
This estimate could be further improved depending on the polynomial~$Q$, as in Corollary 2.11 of~\cite{thermalidsl}. However, for a general given polynomials $P$, it is hard to compute the corresponding exponent~$\delta_P$, thus we use the Riemann-Lebesgue lemma in the last step in the proof of Theorem~\ref{dliuf}. This is why we do not give the precise thermalization decay rates unlike Corollaries 2.9, 2.10 and~2.11 in~\cite{thermalidsl}. Finally, note that although it does not yield an exact formula, it is possible to use the algorithm of Theorem 4.1 in~\cite{calculdensalgo} to approximate the spectral distribution of a given polynomial in semicircular variables.

We conclude this section by summarizing the organization of the paper. In Section~\ref{3definit}, we first recall definitions from random matrix theory and free probability theory. Then in Subsection~\ref{3poly}, we introduce the formalism suited to handle polynomials in non-commutative random variables and their derivatives. In Subsection~\ref{sec:SD}, we recall the Schwinger-Dyson equations for Gaussian Wigner matrices as well as the cumulant expansions which can be viewed as a generalization of the Schwinger-Dyson equations for Wigner matrices. In Section~\ref{sec:concentration}, we derive key concentration estimates for the trace of products of exponentials of polynomials in Wigner matrices. This is accomplished by a continuous interpolation between Wigner matrices and GUE matrices in combination with a recursive moment estimate. In Section~\ref{sec:proof main theorems}, we then combine the concentration estimates from Section~\ref{sec:concentration} with the main results from~\cite{topoexp} to prove our main results in Theorem~\ref{main_result} and Theorem~\ref{dliuf}.

\section{Framework and standard properties}
\label{3definit}

\subsection{Definitions from Random Matrix Theory}

In this subsection we define the main object of our study, the Wigner matrix. We make the assumption that the matrix entries are independent and have finite moments to every order, however we do not need to assume that they have the same law.

\begin{defi}
	\label{defwigner}
	We say that a square random matrix $Y$ of size $N$ is a Wigner matrix if it is a Hermitian or symmetric matrix whose entries are independent up to the Hermitian symmetry and such that for any $i,j$, $\E[Y_{i,j}]= 0$, and additionally if $i\neq j$, $\E[|Y_{i,j}|^2]=N^{-1}$. Besides we assume that for any $p\in \N$,
	\begin{equation}
		\label{assump}
		\sup_{N\in\N, 1\leq i,j\leq N} \E\left[\left| \sqrt{N} Y_{i,j}\right|^p\right] < \infty.
	\end{equation}
\end{defi}

The assumption on the large moments being finite is natural as we are working with polynomials. Although since in Theorem \ref{main_result} we give equations which hold true with high probability rather than estimates on the expectation, there might be a way around this assumption by using the truncation method (see Section 2 of \cite{anderste}), but this would reflect in the error terms of Equations \eqref{trace} and \eqref{scalar}.

Besides, we use two specific types of Wigner matrices whose entries are all Gaussian, the Gaussian Unitary Ensemble (GUE) and the Gaussian Orthogonal Ensemble (GOE).

\begin{defi}
	\label{3GUEdef}
	A GUE random matrix $X^N$ of size $N$ is a Hermitian matrix whose entries are random variables with the following laws:
	\begin{itemize}
		\item For $1\leq i\leq N$, the random variables $\sqrt{N} X^N_{i,i}$ are independent centered Gaussian random variables of 
		variance $1$.
		\item For $1\leq i<j\leq N$, the random variables $\sqrt{2N}\ \Re{X^N_{i,j}}$ and $\sqrt{2N}\ \Im{X^N_{i,j}}$ are independent 
		centered Gaussian random variables of variance $1$, independent of  $\left(X^N_{i,i}\right)_i$.
	\end{itemize}
\end{defi}

\begin{defi}
	\label{3GOEdef}
	A GOE random matrix $X^N$ of size $N$ is a symmetric matrix whose entries are random variables with the following laws:
	\begin{itemize}
		\item For $1\leq i\leq N$, the random variables $\sqrt{\frac{N}{2}} X^N_{i,i}$ are independent centered Gaussian random variables of variance $1$.
		\item For $1\leq i<j\leq N$, the random variables $\sqrt{N} X^N_{i,j}$ are independent centered Gaussian random variables of variance $1$, independent of  $\left(X^N_{i,i}\right)_i$.
	\end{itemize}
\end{defi}

Finally, we conclude this subsection by defining a notation that we will use regularly in the rest of the paper.

\begin{defi}\label{definition of high probability}
	 Given a sequence of random variables $(\mathcal{X}_N,\mathcal{Y}_N)_{N\geq 1}$ as well as a sequence $(\varepsilon_N)_{N\geq 1}$ of non-negative real numbers, we say that with high probability
$$ \mathcal{X}_N = \mathcal{Y}_N + \mathcal{O}(\varepsilon_N),$$
if for any $k>0$, there exists a numerical constant $C$ such that,
$$\P(|\mathcal{X}_N-\mathcal{Y}_N|\geq C\ \varepsilon_N)\leq N^{-k}$$
holds for any $N$ sufficiently large.
\end{defi}

\subsection{Definitions from Free Probability Theory}
\label{3deffree}

In order to be self-contained, we begin by recalling the following definitions from free probability.

\begin{defi}\label{3freeprob}
\phantom{X}	
	\begin{itemize}
		\item A $\mathcal{C}^*$-probability space $(\A,*,\tau,\norm{.}) $ is a unital $\mathcal{C}^*$-algebra $(\A,*,\norm{.})$ endowed with a state $\tau$, i.e. a linear map $\tau : \A \to \C$ satisfying $\tau(1_{\A})=1$ and $\tau(a^*a)\geq 0$ for all $a\in \A$. In this paper we always assume that $\tau$ is a trace, i.e. that it satisfies $\tau(ab) = \tau(ba) $ for any $a,b\in\A$. An element of $\A$ is called a 
		non-commutative random variable. We will always work with a faithful trace, namely, for $a\in\A$, $\tau(a^*a)=0$ if and only if $a=0$.
		
		\item Let $\A_1,\dots,\A_n$ be $*$-subalgebras of $\A$, having the same unit as $\A$. They are said to be free if for all $k$, for all $a_i\in\A_{j_i}$ such that $j_1\neq j_2$, $j_2\neq j_3$, \dots , $j_{k-1}\neq j_k$:
		\begin{equation}
			\label{induction}
			 \tau\Big( (a_1-\tau(a_1))(a_2-\tau(a_2))\cdots (a_k-\tau(a_k)) \Big) = 0.
		\end{equation}
		Families of non-commutative random variables are said to be free if the $*$-subalgebras they generate are free.
		
		\item Let $ A= (a_1,\ldots ,a_k)$ be a $k$-tuple of non-commutative random variables. The joint distribution of the family $A$ is the linear form $\mu_A : P \mapsto \tau\big[ P(A, A^*) \big]$ on the set of polynomials in $2k$ non-commutative variables. By convergence in distribution, for a sequence of families of variables $(A_N)_{N\geq 1} = (a_{1}^{N},\ldots ,a_{k}^{N})_{N\geq 1}$ 
		in $\mathcal C^*$-algebras $\big( \mathcal A_N, ^*, \tau_N, \norm{.} \big)$,
		we mean the pointwise convergence of
		the map 
		$$ \mu_{A_N} : P \mapsto \tau_N \big( P(A_N, A_N^*) \big).$$		
		\item A family of non-commutative random variables $ x=(x_1,\dots ,x_d)$ is called 
		a free semicircular system if the non-commutative random variables are free, 
		self-adjoint ($x_i=x_i^*$), and for all $k\in\N$ and $i\in [1,d]$, one has
		\begin{equation*}
			\tau( x_i^k) =  \int_{\R} t^k d\sigma(t),
		\end{equation*}
		with $d\sigma(t) = \frac 1 {2\pi} \sqrt{4-t^2} \ \mathbf 1_{|t|\leq2} \ dt$ the semicircle distribution.
		
	\end{itemize}
	
\end{defi}

It is important to note that thanks to \cite[Theorem 7.9]{nica_speicher}, which we recall next, one can consider free copies of any non-commutative random variable.

\begin{theorem}
	\label{3freesum}
	
	Let $(\A_i,\phi_i)_{i\in I}$ be a family of $\mathcal{C}^*$-probability spaces such that the functionals $\phi_i : \A_i\to\C$, $i\in I$, are faithful traces. Then there exist a $\mathcal{C}^*$-probability space $(\A,\phi)$ with $\phi$ a faithful trace, and a family of norm-preserving unital $*$-homomorphism $W_i: \A_i\to\A$, $i\in I$, such that:
	
	\begin{itemize}
		\item $\phi \circ W_i = \phi_i$, $\forall i \in I$.
		\item The unital $\mathcal{C}^*$-subalgebras $W_i(\A_i)$, $i\in I$, form a free family in $(\A,\phi)$.
	\end{itemize}
	We will usually denote $\A$ by $\underset{i\in I}{*}\A_i$ or simply $\A_1*\A_2$ when $I$ only has two elements.
\end{theorem}

Let us finally fix a few notations concerning the spaces and traces that we use in this paper.

\begin{defi}\phantom{X}	
	\label{3tra}
	\begin{itemize}
		\item $\M_N(\C)$ is the set of $N\times N$ matrices with coefficients in $\C$.
		\item $(\A_N,\tau_N)$ is the free product $\M_N(\C) * \mathcal{C}_d$ of $\M_N(\C)$ with $\mathcal{C}_d$ the $\mathcal{C}^*$-algebra generated by a free semicircular system $x=(x_1,\dots,x_d)$, that is the $\mathcal{C}^*$-probability space built in Theorem \ref{3freesum}. Note that when restricted to $\M_N(\C)$, $\tau_N$ is just the regular normalized trace on matrices, in this case we will denote it by $\ts_N$. The restriction of $\tau_{N}$ to $\mathcal{C}_d$ is denoted as $\tau$. Note that one can view this space as the limit of a matrix space, we refer to Proposition 3.5 from \cite{un}. 
		\item $\tr_N$ is the non-normalized trace on $\M_N(\C)$, while $\ts_N$ is the normalized one.
		\item We denote $E_{r,s}\in\M_N(\C)$ the matrix with entries equal to $0$ except in $(r,s)$ where it is equal to $1$. 
	\end{itemize}
\end{defi}

In order to interpret Theorem \ref{main_result}, we need to define the conditional expectation first.

\begin{defi}
	\label{conditionalmanqu}
	Let $\A$ be a $\mathcal{C}^*$-algebra and $\B$ be a unital $\CC^*$-algebra, then $E:\A\to\B$ is said to be a conditional expectation if it is a linear map such that:
	\begin{itemize}
		\item E is positive, i.e. for every $a\in\A$, there exist $b\in\B$ such that $E[a^*a] = b^*b$.
		\item E is a $\B$-$\B$-bimodule map, i.e. for every $a\in\A, b_1,b_2\in\B$, one has $E[b_1 a b_2] = b_1 E[a] b_2$.
		\item For every $b\in\B$, $E[b]=b$.
	\end{itemize}
\end{defi}

It is well-known that we have the following property.

\begin{prop}
	\label{condexp}
	Let $(\A,\tau)$ be a $W^*$-algebra, i.e. a von Neumann algebra with a faithful normal trace $\tau$, and $\B$ a $W^*$-subalgebra. Then there exists a unique conditional expectation $E:\A\to\B$. Besides given $a\in\A$, it is characterized by the fact that for any $b\in\B$, $\tau(E[a]b) = \tau(ab)$. Finally, for every element $a\in\A$, we have that $\norm{E[a]}\leq \norm{a}$, where $\norm{\cdot}$ is the operator norm.
\end{prop}

With those definitions and properties, one can deduce the following which gives us some understanding on Equation \eqref{scalar}.

\begin{rem}
\phantom{X}		
	\begin{itemize}
		
	\item
	While the $\CC^*$-algebra $\A_N$ built in Definition \ref{3tra} is not necessarily a von Neumann algebra, one can always replace it by its enveloping von Neumann algebra. Hence, since $\M_N(\C)$ is a von Neumann algebra, there exists a unique conditional expectation $E_{\M_N(\C)}$ from $\A_N$ to $\M_N(\C)$. That being said, it is not really necessary to understand the methods used to define the conditional expectation to read this paper. The rest of this remark will provide more details on how to estimate this quantity most of the time.
		
	\item
	Thanks to Proposition \ref{condexp} we can deduce several properties of the conditional expectation in Equation \eqref{scalar}, first and foremost one has the following bound
	\begin{align*}
		&\left|\left\langle \x, E_{\M_N(\C)}\big[f_1(P_1(x,A^N)) \cdots f_k(P_k(x,A^N))\big]\y \right\rangle\right| \\
		&\quad\leq \norm{E_{\M_N(\C)}\big[f_1(P_1(x,A^N)) \cdots f_k(P_k(x,A^N))\big]} \norm{\x}_2\norm{\y}_2 \\
		&\quad\leq \norm{f_1(P_1(x,A^N))} \cdots \norm{f_k(P_k(x,A^N))} \norm{\x}_2\norm{\y}_2. \\
	\end{align*}
	Besides if $f_i=\id_{\R}$, then $\norm{f_i(P_i(x,A^N))} = \norm{P_i(x,A^N)}$, and otherwise
	$$ \norm{f_i(P_i(x,A^N))} \leq \sup_{x\in\R} |f(x)| \leq \int_{\R}1\ d|\mu_i| \leq \norm{f_i}_4.$$
	Consequently as long as our functions do not depend on $N$, the first order term in Equation \eqref{scalar} is also independent of $N$ except with respect to the norm of the matrices $A^N$ and the one of the vectors $\x,\y$ which is to be expected.
	
	\item
	It is also worth noting that since 
	\begin{align*}
		&\left\langle \x, E_{\M_N(\C)}\big[f_1(P_1(x,A^N)) \cdots f_k(P_k(x,A^N))\big]\y \right\rangle \\
		&\quad= \tr_N\left( E_{\M_N(\C)}\big[f_1(P_1(x,A^N)) \cdots f_k(P_k(x,A^N))\big] \y\x^* \right) \\
		&\quad= N \tau_N\left( f_1(P_1(x,A^N)) \cdots f_k(P_k(x,A^N))\ \y\x^* \right),
	\end{align*}
	it is not necessary to compute the conditional expectation of $f_1(P_1(x,A^N)) \cdots f_k(P_k(x,A^N))$ to determine the first order term of Equation~\eqref{scalar}. However, the formulation with the conditional expectation renders it clear that this term is of order one.
	
	\item
	It is possible to use advanced combinatorics associated with free probability to compute the leading terms on the right of~\eqref{trace} and~\eqref{scalar}, especially when for every $i$ such that $f_i \neq \id_{\R}$, $P_i$ is a polynomial in $x$ rather than $(x,A^N)$; see Section 2 of \cite{thermalidsl}. For example with $s$ a single semicircular variable, we have that
	$$  \tau_N\left( f_1(s)A_1^N \cdots f_k(s)A_k^N \right) = \sum_{\pi \in \mathrm{NC}[k]} \langle A_1^N,\dots,A_k^N\rangle_{K(\pi)} \prod_{B\in\pi} \mathrm{sc}_{\circ}[B],  $$
	where $\mathrm{NC}[k]$ is the set of non-crossing partitions of $ \llbracket 1,k \rrbracket $, $\langle A_1^N,\dots,A_k^N\rangle_{K(\pi)}$ is a quantity that depends only on the matrices $A_1^N,\dots,A^N_k$ and the Kreweras complement $K(\pi)$ of $\pi$, whereas the free cumulant function $\mathrm{sc}_{\circ}[\;\cdot\; ]$ depends on the functions $f_1,\dots,f_k$ and the semicircular distribution. We refer to~\cite{nica_speicher} for an introduction to the combinatorics of free probability.
	\item
	If one does not have any matrices $A^N$, then since $f_1(P_1(x)) \cdots f_k(P_k(x))$ are free from $\M_N(\C)$, we have that for any $b\in\M_N(\C)$,
	$$\tau_{N}\big(f_1(P_1(x)) \cdots f_k(P_k(x)) b\big) = \tau_{N}(b) \times \tau\big(f_1(P_1(x)) \cdots f_k(P_k(x))\big).$$
	Hence, $E_{\M_N(\C)}\big[f_1(P_1(x)) \cdots f_k(P_k(x))\big] = \tau\big(f_1(P_1(x)) \cdots f_k(P_k(x))\big)$ and 
	$$ \left\langle \x, E_{\M_N(\C)}\big[f_1(P_1(x)) \cdots f_k(P_k(x))\big]\y \right\rangle =  \langle \x,\y\rangle\times \tau\big(f_1(P_1(x)) \cdots f_k(P_k(x))\big).$$
	\end{itemize}

\end{rem}

\subsection{Noncommutative polynomials and derivatives}
\label{3poly}

Let $\A_{d,2r}=\C\langle X_1,\dots,X_d,Z_1,\dots,Z_{2r}\rangle$ be the set of non-commutative polynomials in $d+2r$ variables. We set $q=2r$ to simplify notations. We denote by $\deg M$ the total degree of $M$ (that is the sum of its degree in each letter $X_1,\dots,X_d,Z_1,\dots,Z_{2r}$). Let us now define several maps that are frequently used in Operator Algebra. First, for $A,B,C\in \A_{d,q}$, let
\begin{equation}
	\label{3defperdu}
	A\otimes B \# C := ACB,\quad  A\otimes B \widetilde{\#} C := BCA,\quad  m(A\otimes B) := BA.
\end{equation}

\noindent We define an involution $*$ on $\A_{d,q}$ by $X_i^* = X_i$, $Z_i^* = Z_{i+r}$ if $1\leq i\leq r$, $Z_i^* = Z_{i-r}$ else, and then we extend it to $\A_{d,q}$ by linearity and the formula $(\alpha P Q)^* = \overline{\alpha} Q^* P^*$. $P\in \A_{d,q}$ is said to be self-adjoint if $P^* = P$.  Self-adjoint polynomials have the property that if $x_1,\dots,x_d,z_1,\dots,z_r$ are elements of a $\mathcal{C}^*$-algebra such that $x_1,\dots,x_d$ are self-adjoint, then so is $P(x_1,\dots,x_d,z_1,\dots,z_r,z_1^*,\dots,z_r^*)$. 

\begin{defi}
	\label{3application}
	
	If $1\leq i\leq d$, one defines $\partial_i: \A_{d,q} \longrightarrow \A_{d,q} \otimes \A_{d,q}$ the non-commutative derivative with respect to $X_i$ as such. First on a monomial  $M\in \A_{d,q}$, one sets
	$$ \partial_i M = \sum_{A,B\in\A_{d,q} \text{ such that } M=AX_iB} A\otimes B,$$
	and then extend it by linearity to all polynomials. We can also define $\partial_i$ by induction with the formulas,
	\begin{equation}
		\label{3leibniz}
		\begin{array}{ccc}
			&\forall P,Q\in \mathcal{A}_{d,q},\quad \partial_i (PQ) = \big(\partial_i P\big) \big( 1\otimes Q \big) + \big(P\otimes 1\big) \big(\partial_i Q\big) , \\
			& \\
			&\forall i,j,\quad \partial_i X_j = \delta_{i,j} 1\otimes 1,\quad \partial_i Z_j=0. \end{array}
	\end{equation}
	Similarly, with $m$ as in \eqref{3defperdu}, one defines the cyclic derivative  $D_i: \A_{d,q} \longrightarrow \A_{d,q}$ for $P\in \A_{d,q}$ by
	$$ D_i P = m\circ \partial_i P \ . $$
	
\end{defi}

In this paper we need to work not only with polynomials but also with more general functions, since we will work with the Fourier transform we introduce the following space.

\begin{defi}
	We set $\mathcal{S} = \left\{ R\in \mathcal{A}_{d,q}\ |\ R^*=R \right\}$, then we denote
	$$\mathcal{F}_{d,q} = \C\big\langle (E_R)_{R\in\S}, X_1,\dots,X_d,Z_1,\dots,Z_{2r}\big\rangle,$$
	the set of polynomials in $ X_1,\dots,X_d,Z_1,\dots,Z_{2r}$ as well as in a family of variable indexed by $\S$. Then given  $ y = (x_1,\dots,x_d,z_1,\dots,z_r,z_1^*,\dots,z_r^*)$ elements of a $\CC^*$-algebra, one can define by induction the evaluation of an element of $\F_{d,q}$ in $y$ by following the following rules:
	\begin{itemize}
		\item $\forall Q\in\A_{d,q}$, $Q(y)$ is defined as usual,
		\item $\forall Q_1,Q_2\in \F_{d,q}$, $(Q_1+Q_2)(y)= Q_1(y)+Q_2(y)$, $(Q_1Q_2)(y)= Q_1(y)Q_2(y)$,
		\item $\forall R\in\S$, $E_R(y) = e^{\i R(y)}$.
	\end{itemize}
	One can extend the involution $*$ from $\A_{d,q}$ to $\mathcal{F}_{d,q}$ by setting $(E_R)^* = E_{(-R)} $, and then again we have that if $Q\in\F_{d,q}$ is self-adjoint, then so is $Q(y)$. Finally, in order to make notations more transparent, we will usually write $e^{\i R}$ instead of $E_R$.
\end{defi}
Note that for technical reasons which are explained in Remark 2.10 of \cite{topoexp}, one cannot view $\mathcal{F}_{d,q}$ as a subalgebra of the set of formal power series in $X_1,\dots,X_d,Z_1,\dots,Z_{2r}$. Hence, why we need to introduce the notation $E_R$.

Now, as we will see in Proposition \ref{3SD}, a natural way to extend the definition of $\partial_i$ (and $D_i$) to $\F_{d,q}$ is by setting
\begin{equation}
	\label{3ext}
	\partial_i e^{\i R} = \i \int_0^1 \big(e^{\i \alpha R}\otimes 1\big)\ \partial_i R\ \big(1\otimes e^{\i (1-\alpha) R}\big) d\alpha .
\end{equation}

\noindent However we cannot define the integral properly on $\mathcal{F}_{d,q}\otimes \mathcal{F}_{d,q}$, but after evaluating our polynomials in a matrix space this is not a problem anymore. Indeed, a tensor of matrix spaces is also a matrix space and hence the integral of continuous functions on this space is well-defined. Thus we define the non-commutative differential on $\F_{d,q}$ as following.

\begin{defi}
	\label{3technicality}
	For $\alpha\in [0,1]$, we define $\partial_{\alpha,i}: \F_{d,q}\to \F_{d,q}\otimes \F_{d,q}$ as the map which satisfies \eqref{3leibniz} and is such that for any $R\in \mathcal{A}_{d,q}$ self-adjoint,
 	$$ \partial_{\alpha,i} e^{\i R} = \i \left(e^{\i \alpha R}\otimes 1\right)\ \partial_i R\ \left(1\otimes e^{\i (1-\alpha) R}\right),\quad D_{\alpha,i} = m \circ \partial_{\alpha,i} . $$
	Then, given $ y = (y_1,\dots, y_{d+q})$ elements of $M_N(\C)$, we define for any $Q\in \F_{d,q}$,
	$$ \partial_i Q(y) = \int_{0}^1 \partial_{\alpha,i} Q(y)\ d\alpha,\quad D_i Q(y) = \int_{0}^1 D_{\alpha,i} Q(y)\ d\alpha . $$
\end{defi}

\noindent Note that for any $P\in\A_{d,q}$, since $\int_0^11d\alpha = 1$, we do also have with $\partial_i Q$ defined as in Definition~\ref{3application}, that 
$$ \partial_i Q(y) = \int_{0}^1 \partial_{\alpha,i} Q(y)\ d\alpha .$$
Thus Definition \ref{3technicality} extends indeed the definition of $\partial_i$ from $\A_{d,q}$ to $\F_{d,q}$. Besides, it also means that we can rigorously define the composition of noncommutative differentials. Since the map $\partial_{\alpha,i}$ goes from $\F_{d,q}$ to $\F_{d,q}\otimes \F_{d,q}$ it is very easy to do so. For example one defines the following operator that we use later on.

\begin{defi}
	\label{3operatordef}
	Let $Q\in \F_{d,q}$, given $ y = (y_1,\dots, y_{d+q})$ elements of a $\CC^*$-algebra, let $i\in [1,d]$, with $\circ$ the composition of operator we define
	$$ \partial_i\circ D_i Q(y) = \int_{[0,1]^2} \partial_{\alpha_2,i}\circ D_{\alpha_1,i} Q(y)\ d\alpha_1 d\alpha_2 .$$
\end{defi}

\subsection{The Schwinger-Dyson equations and their generalization}
\label{sec:SD}

A tool that we use repeatedly with Gaussian random matrices are the so-called Schwinger-Dyson equations. It is a consequence of Gaussian integration by parts which can be summarized into the following formula. If $Z$ is a centered Gaussian random variable with variance one and $f$ a $\mathcal{C}^1$ function, then
\begin{equation}
	\label{3IPPG}
	\E[Z f(Z)] = \E[\partial_Z f(Z)] \ .
\end{equation}
From there on, we deduce the Schwinger-Dyson equations in the following proposition. For more information about these equations and their applications, we refer to \cite{alice}, Lemma 5.4.7.

\begin{prop}
	\label{3SD}
	Let $X^N$ be a GOE random matrix of size $N$, $A^N$ deterministic matrices, $Q\in \F_{1,q}$, then
	\begin{equation}
		\label{SDEPC}
		\E\Big[\ts_N\left(X^N\ Q(X^N,A^N)\right) \Big] = \E\Big[ \ts_N^{\otimes 2} \left(\partial_1 Q(X^N,A^N)\right) \Big] + \frac{1}{N} \E\Big[ \ts_N\left(h(\partial_1 Q(X^N,A^N))\right) \Big],
	\end{equation}
	where $h$ is the linear map such that $h(A\otimes B) = A^TB$ with $A^T$ is the transpose of $A$.
\end{prop}

\begin{proof}
	Let us first assume that $Q\in\A_{1,q}$. One can write $X^N = \frac{1}{\sqrt{N}} (x_{r,s})_{1\leq r,s\leq N}$ and thus
	
	\begin{align*}
		\E\Big[ \ts_N(X^N\ Q(X^N)) \Big] &= \frac{1}{N^{3/2}} \sum_{r,s} \E\left[ x_{r,s}\ \tr_N(E_{r,s}\ Q(X^N,A^N)) \right] \\
		&= \frac{1}{N^{3/2}} \sum_{r,s} \E\left[ \tr_N(E_{r,s}\ \partial_{x_{r,s}} Q(X^N,A^N)) \right] \\
		&= \frac{1}{N^{2}} \sum_{r,s} \E\left[ \tr_N\left(E_{r,s}\ \partial_1 Q(X^N,A^N) \#( E_{s,r}+E_{r,s})\right) \right] \\
		&= \E\Big[ \ts_N^{\otimes 2} (\partial_1 Q(X^N,A^N)) \Big] + \frac{1}{N} \E\Big[ \ts_N\left(h(\partial_1 Q(X^N,A^N))\right) \Big],
	\end{align*}
	
	\noindent where notably we used that for any matrices $A,B\in\M_N(\C)$,
	$$\sum_{1\leq i,j\leq N} A_{i,j}B_{i,j} = \tr_{N}(A^TB).$$
	If $Q\in\F_{d,q}$, then the proof is pretty much the same but we need to use Duhamel's formula (for a very similar proof see \cite{deux}, Proposition 2.2) which states that for any matrices $A$ and $B$, 
	\begin{equation}
		\label{3duha}
		e^B - e^A = \int_{0}^1 e^{\alpha B} (B-A) e^{(1-\alpha)A}\ d\alpha .
	\end{equation}
	Thus this allows us to prove that for any self-adjoint polynomials $P\in\A_{d,q}$,
	$$ \partial_{x_{r,s}} e^{\i P(X^N)} = \i \int_{0}^1 e^{\i \alpha P(X^N)}\ \partial_1 P(X^N) \# E_{s,r}\ e^{\i (1-\alpha) P(X^N)}\ d\alpha,$$
	and the conclusion follows.
	\end{proof}

In the case of GUE random matrices we have an even shorter formula. The following Proposition is actually Proposition 2.23 of \cite{topoexp} whose proof is quite similar to the one of Proposition \ref{3SD}.
\begin{prop}
	\label{skscksd}
	Let $X^N$ be a GUE random matrix of size $N$, $A^N$ deterministic matrices, $Q\in \F_{1,q}$, then
	\begin{equation}
		\E\Big[\ts_N\left(X^N\ Q(X^N,A^N)\right) \Big] = \E\Big[ \ts_N^{\otimes 2} \left(\partial_1 Q(X^N,A^N)\right) \Big].
	\end{equation}
\end{prop}

For studying Wigner matrices, there exists a more general formula, the cumulant expansion. Its usefulness in random matrix theory was recognized in~\cite{KKP} and has widely been used since, {\it e.g.}~\cite{BoutetdeMonvel,ErdoesKruegerSchroeder,moment,He+Knowles,HLY,LeeSchnellisparse,lytova_pastur}. We will use a specific version with an explicit expression for the remainder. To do so we follow the proof of Proposition 3.1 of \cite{lytova_pastur} but we do not upper bound the remainder.

\begin{prop}\label{cuulantmexpansion}
	Let $u,v$ be real random variables such that $\E[|u|^{\ell+2}]< \infty$ and $\E[|v|^{\ell+2}]< \infty$ for some natural number $\ell$. Let  $(\kappa_{n,m})_{n,m\in\N}$ be the joint cumulants of $u$ and $v$, i.e. the numbers that satisfy
	$$ \log\E\left[ e^{\i t u + \i s v} \right] = \sum_{0\leq n+m \leq \ell} \frac{\kappa_{n,m}}{n! m!} (\i t)^n (\i s)^m + o (\max\{|t|,|s|\}^\ell).$$
	Then for any function $\Phi: \R^2 \to\C$ of the class $\mathcal{C}^{\ell+1}$, we have that 
	\begin{equation}
		\label{kdkc}
		\E\left[ u \Phi(u,v)\right] = \sum_{\substack{0\leq a+b\leq \ell}} \frac{\kappa_{a+1,b} }{a! b!}\times \E[\partial_u^a\partial_v^b \Phi(u,v)] + \varepsilon_{\ell+1} ,
	\end{equation}
	with
	$$ \varepsilon_{\ell+1} = \E\left[ u g_{0,0}(u,v)\right] - \sum_{\substack{0\leq a+b \leq \ell}} \frac{\kappa_{a+1,b} }{a! b!}\times \E[g_{a,b}(u,v)],$$
	where we defined 
	\begin{align*}
		g_{a,b}:(X_1,X_2)\in\R^2 \longmapsto \frac{1}{(\ell-a-b)!} \sum_{\varepsilon\in\{1,2\}^{\ell+1-a-b}}& X_{\varepsilon_1}\cdots X_{\varepsilon_{\ell+1-a-b}} \\
		&\int_0^1 \partial_{\varepsilon_1}\cdots \partial_{\varepsilon_{\ell+1-a-b}}\partial_{1}^a\partial_{2}^b\Phi(tX_1,tX_2)\ (1-t)^{\ell-a-b}\ dt.
	\end{align*}
\end{prop}

\begin{proof}
	Since we have that
	$$ \E\left[ u e^{\i t u + \i s v} \right] = \E\left[ e^{\i t u + \i s v} \right] \times (-\i \partial_t)\left(\log\E\left[ e^{\i t u + \i s v} \right]\right),$$
	then with $\mu_{n,m} = \E[u^n v^m]$, this implies that for $n+m\leq \ell$,
	$$\mu_{n+1,m} = \sum_{\substack{0\leq a\leq n \\ 0\leq b\leq m}} \frac{\kappa_{a+1,b}}{a! b!}\times \frac{n! m!}{ (n-a)! (m-b)!}\mu_{n-a,m-b}. $$
	Consequently, for any polynomial $P$ of degree smaller than $\ell$, we have that
	\begin{align*}
		\E[u P(u,v)] &= \sum_{\substack{0\leq a< \infty \\ 0\leq b< \infty}} \frac{\kappa_{a+1,b} }{a! b!}\times \E[\partial_u^a\partial_v^b P(u,v)] \\
		&= \sum_{\substack{0\leq a+b \leq \ell}} \frac{\kappa_{a+1,b} }{a! b!}\times \E[\partial_u^a\partial_v^b P(u,v)].
	\end{align*}
	For $X_1,X_2\in\R$, let $f:t\in [0,1]\mapsto \Phi(tX_1,tX_2)$, then thanks to Taylor's theorem, we have that
	$$ f(1) = \sum_{j=0}^{\ell} \frac{f^{(j)}(0)}{j!} + \frac{1}{\ell!}\int_{0}^{1} f^{(\ell+1)}(t) (1-t)^{\ell} dt.$$
	Consequently, there exists a polynomial $\pi_\ell$ of degree at most $\ell$ such that
	\begin{align}
		\label{differenlate}
		\Phi(X_1,X_2) &= \pi_\ell(X_1,X_2) + \frac{1}{\ell!} \sum_{\varepsilon\in\{1,2\}^{\ell+1}} X_{\varepsilon_1}\cdots X_{\varepsilon_{\ell+1}}\int_0^1 \partial_{\varepsilon_1}\cdots \partial_{\varepsilon_{\ell+1}}\Phi(tX_1,tX_2)\ (1-t)^\ell\ dt \\
		&= \pi_\ell(X_1,X_2) + g_{0,0}(X_1,X_2). \nonumber
	\end{align}
	
	\noindent Thus we get that
	\begin{align*}
		\E\left[ u \Phi(u,v)\right] &= \E\left[ u \pi_\ell(u,v)\right] + \E\left[ u g_{0,0}(u,v)\right] \\
		&=  \sum_{\substack{0\leq a+b \leq \ell}} \frac{\kappa_{a+1,b} }{a! b!}\times \E[\partial_u^a\partial_v^b \pi_\ell(u,v)] + \E\left[ u g_{0,0}(u,v)\right].
	\end{align*}
	But thanks to Equation \eqref{differenlate}, we have that 
	\begin{align*}
		&\partial_1\left( \Phi(X_1,X_2) - \pi_\ell(X_1,X_2) \right) \\ 
		&= \frac{1}{\ell!} \sum_{\substack{\varepsilon\in\{1,2\}^{\ell+1} \\ j\in[1,\ell+1]}} X_{\varepsilon_1}\cdots X_{\varepsilon_{j-1}} \partial_1(X_{\varepsilon_{j}}) X_{\varepsilon_{j+1}}\cdots X_{\varepsilon_{\ell+1}}\int_0^1 \partial_{\varepsilon_1}\cdots \partial_{\varepsilon_{\ell+1}}\Phi(tX_1,tX_2)\ (1-t)^\ell\ dt \\ 
		&\quad + \frac{1}{\ell!} \sum_{\varepsilon\in\{1,2\}^{\ell+1}} X_{\varepsilon_1}\cdots X_{\varepsilon_{\ell+1}}\int_0^1 \partial_{\varepsilon_1}\cdots \partial_{\varepsilon_{\ell+1}}\partial_1\Phi(tX_1,tX_2)\times t(1-t)^\ell\ dt \\
		&= \frac{1}{\ell!} \sum_{\varepsilon\in\{1,2\}^\ell} X_{\varepsilon_1}\cdots X_{\varepsilon_{\ell}}\int_0^1 \partial_{\varepsilon_1}\cdots \partial_{\varepsilon_{\ell}} \partial_1\Phi(tX_1,tX_2)\ (1-t)^\ell\ dt\times (\ell+1) \\
		&\quad + \frac{1}{\ell!} \sum_{\varepsilon\in\{1,2\}^{\ell}} X_{\varepsilon_1}\cdots X_{\varepsilon_{\ell}}\int_0^1 \frac{d}{dt}\left(\partial_{\varepsilon_1}\cdots \partial_{\varepsilon_{\ell}}\partial_1\Phi(tX_1,tX_2)\right)\times t(1-t)^\ell\ dt \\
		&= \frac{1}{\ell!} \sum_{\varepsilon\in\{1,2\}^\ell} X_{\varepsilon_1}\cdots X_{\varepsilon_{\ell}}\int_0^1 \partial_{\varepsilon_1}\cdots \partial_{\varepsilon_{\ell}} \partial_1\Phi(tX_1,tX_2)\times (1-t)^{\ell-1} \times((1-t)(\ell+1) - (1-t) + \ell t )\ dt \\
		&= \frac{1}{(\ell-1)!} \sum_{\varepsilon\in\{1,2\}^\ell} X_{\varepsilon_1}\cdots X_{\varepsilon_{\ell}}\int_0^1 \partial_{\varepsilon_1}\cdots \partial_{\varepsilon_{\ell}} \partial_1\Phi(tX_1,tX_2)\times (1-t)^{\ell-1}\ dt.
	\end{align*}
	
	\noindent Hence, by induction we get that $ \partial_1^a\partial_2^b\left( \Phi- \pi_\ell \right) = g_{a,b}$. Note that we used that $\Phi$ was of class $\mathcal{C}^{\ell+2}$ in this computation but by an argument of density we only need $\Phi$ to be of class $\mathcal{C}^{\ell+1}$ eventually. Thus in conclusion,
	\begin{align*}
		\E\left[ u \Phi(u,v)\right] &= \sum_{\substack{0\leq a+b \leq \ell}} \frac{\kappa_{a+1,b} }{a! b!}\times \E[\partial_u^a\partial_v^b \Phi(u,v)] + \E\left[ u g_{0,0}(u,v)\right] - \sum_{\substack{0\leq a+b \leq \ell}} \frac{\kappa_{a+1,b} }{a! b!}\times \E[g_{a,b}(u,v)]. \\
	\end{align*}
	This proves Equation \eqref{kdkc}.
\end{proof}

\section{Reduction of the problem to the case of GUE matrices}
\label{sec:concentration}

\subsection{Concentration of the trace}

In order to prove Theorem \ref{main_result}, the first step is to reduce the proof to the case of GUE random matrices. To do so we prove two concentration results. First for the trace in Proposition \ref{concentrtrace}, then for the scalar product in Proposition~\ref{concentrcoeff}. Before giving those propositions, we state the following lemmas which will be useful for the proof.

\begin{lemma}
\label{podvfldmvd}
    Given non-negative random variables $\mathcal{X},\mathcal{Y}$, $\alpha\geq 0$, $n\in\N^*$ then with $g\in [1,n]$,
    $$ \alpha^g\E\left[ \mathcal{X}^{n-g} \mathcal{Y} \right] \leq \E\left[\mathcal{X}^n \right] + \alpha^n \E\left[ \mathcal{Y}^{\frac n g} \right].$$
\end{lemma}

\begin{proof} For $a,b\geq 0$ and conjugate exponents $p,q$, we have Young's inequality
    \begin{equation}
    \label{sidcvsokcm}
        ab \leq \frac{a^q}{q} + \frac{b^p}{p}.
    \end{equation}
    Consequently, with $p= \frac{n}{n-g}$, $ q= \frac{n}{g}$, we have that
    \begin{align*}
        \alpha^g\E\left[ \mathcal{X}^{n-g} \mathcal{Y} \right] &\leq \frac{n-g}{n} \E\left[ \mathcal{X}^{n} \right] + \frac{g}{n} \alpha^n \E\left[ \mathcal{Y}^{\frac n g} \right]\leq \E\left[ \mathcal{X}^{n} \right] + \alpha^n \E\left[ \mathcal{Y}^{\frac n g} \right],
    \end{align*}
    and the claim follows
    
\end{proof}

\begin{lemma}
\label{docsocdssdc}
Given
\begin{itemize}
		\item $Y^N$ a $d$-tuple of independent Wigner matrices as in Definition \ref{defwigner}, 
		\item $A^N$ deterministic matrices such that $\sup_{1\leq i\leq q,N\in\N^*} \norm{A_i^N} <\infty$,
        \item $Q_1,\dots,Q_l\in\F_{d,q}$,
        \item $c\in\R^+$,
\end{itemize}
then 
\begin{equation}
    \sup_{N\in\N^*} \E\left[ \left| \ts_N\left(Q_1(Y^N,A^N)\right) \dots\ts_N\left(Q_l(Y^N,A^N)\right) \right|^c \right] <\infty
\end{equation}
Besides if one writes each $Q_j$ as a linear combination of terms of the form $e^{\i P_1}R_1\cdots e^{\i P_k}R_k$, then the upper bound on the equation above does not depend on the (self-adjoint) polynomials $P_1,\dots,P_k$.
\end{lemma}

\begin{proof}

To begin with, thanks to Hölder's inequality, one can always assume that $l=1$. Besides, one has for any $p\geq 1$,
$$ \E\left[ \left| \ts_N\left(Q_1(Y^N,A^N)\right) \right|^{c} \right] \leq \E\left[ \left| \ts_N\left(Q_1(Y^N,A^N)\right) \right|^{cp} \right]^{1/p}.$$
Consequently, one can assume that $c$ is an even integer. In which case we have that
$$ \left| \ts_N\left(Q_1(Y^N,A^N)\right) \right|^{c} \leq \ts_N\left(|Q_1(Y^N,A^N)|^c\right) = \ts_N\left(\left(Q_1(Y^N,A^N)^*Q_1(Y^N,A^N)\right)^{c/2}\right). $$
Hence, it is sufficient to prove that for a given $Q\in\F_{d,q}$,
\begin{equation}
    \sup_{\N\in\N^*} \left| \E\left[ \ts_N\left(Q(Y^N,A^N)\right) \right]\right| <\infty.
\end{equation}
By linearity, one can also assume that there exists $P_1,\dots,P_k,R_1,\dots,R_k \in\A_{d,q}$ non-commutative polynomials such that $P_1,\dots,P_k$ are self-adjoint and $Q= e^{\i P_1}R_1\cdots e^{\i P_k}R_k$. Then thanks once again to Hölder's inequality, one has that
$$ \left| \E\left[ \ts_N\left(Q(Y^N,A^N)\right) \right]\right| \leq \prod_{j=1}^k \E\left[ \ts_N\left(\left(R_j(Y^N,A^N)^*R_j(Y^N,A^N)\right)^k\right) \right]^{\frac{1}{2k}} \E\left[ \ts_N\left(\left|e^{\i P_j(Y^N,A^N)}\right|^{2k}\right) \right]^{\frac{1}{2k}}.$$
Let us now remark that given a self-adjoint matrix $T$, we have that $\left|e^{\i T}\right|=\sqrt{e^{-\i T}e^{\i T}}=I_N$, Thus
$$ \left| \E\left[ \ts_N\left(Q(Y^N,A^N)\right) \right]\right| \leq \prod_{j=1}^k \E\left[ \ts_N\left(\left(R_j(Y^N,A^N)^*R_j(Y^N,A^N)\right)^k\right) \right]^{\frac{1}{2k}}.$$
Thanks to Theorem 5.4.5 of \cite{alice}, such a quantity is uniformly bounded with respect to $N$ (since we assumed that the norm of our deterministic matrices was uniformly bounded with respect to $N$).

\end{proof}

We can now state our concentration estimate for the trace.

\begin{prop}
	
	\label{concentrtrace}
	
	Given 
	\begin{itemize}
		\item $Y^N$ a $d$-tuple of independent Wigner matrices as in Definition \ref{defwigner}, 
		\item $X^N$ a $d$-tuple of independent GUE random matrices, independent from $Y^N$,
		\item $A^N$ deterministic matrices such that $\sup_{1\leq i\leq q,N\in\N^*} \norm{A_i^N} <\infty$.
	\end{itemize}
    Let $y_i\in\R$, $y = 1 + \max_i |y_i|$, $P_1,\dots,P_k,R_1,\dots,R_k \in\A_{d,q}$ non-commutative polynomials. If we assume that $P_1,\dots,P_k$ are self-adjoint, then with $Q= e^{\i P_1 y_1}R_1\cdots e^{\i P_k y_k}R_k$, we have for any $\varepsilon>0$ that with high probability, 
	\begin{align}
	\label{lkeshfs}
		&\ts_{N}\left( Q(Y^N,A^N) \right) = \E\left[ \ts_{N}\left(Q(X^N,A^N) \right) \right] + \mathcal O\left( N^{\varepsilon} \frac{y^4}{N} \right) .
	\end{align}
	Besides if $Y^N$ a $d$-tuple of independent GOE or GUE matrices, then one can replace $y^4$ by $y^2$ in the previous equality.
	
\end{prop}

In order to render the proof easier to read, we divide it in five steps. The first one establishes an equation on the moments of $M_N$ (defined in Equation \eqref{elkcmlksc} below). Then in each of the following three steps we bound a specific error term. In the last step, we study the case of the GUE and the GOE.

\begin{proof}[Proof of Proposition~\ref{concentrtrace}]
	\textbf{Step 1 (Using the cumulant expansion):} Our first step is to define
    \begin{equation}
        \label{elkcmlksc}
        M_N := \ts_{N}\Big(Q(Y^N,A^N) - \E[Q(X^N,A^N)]\Big) .
    \end{equation}
	Then we want to study the moments of $M_N$ in order to use Markov's inequality. To do so, we are going to prove that for any $n\in\N$, for any $\varepsilon>0$, there exists a constant $C_{n}$ such that for $N$ large enough,
	\begin{equation}
		\label{indu1}
		\E\left[ \left| M_N \right|^{2n} \right] \leq C_{n} \left(\frac{y^{4}}{N^{1-\varepsilon}}\right)^{2n}.
	\end{equation}
	Let us now set for $t\in\R^+$,
    \begin{equation}
    \label{dfvkdnc}
        X_t^N = \Big(Y^N e^{-t/2} + X^N (1-e^{-t})^{1/2}, A^N\Big),
    \end{equation}
	so that we have the following equality,
	\begin{align*}
		&\frac{d}{dt} \E\left[ \ts_N\left( Q(X_t^N) \right)\times  M_N^{n-1} \overline{M_N}^{n} \right] \\
		&= - \frac{e^{-t}}{2} \sum_{1\leq s\leq d} \E\left[ \ts_N\left( D_sQ(X_t^N)\left( e^{t/2} Y_s^N- \frac{1}{(1-e^{-t})^{1/2}}X_s^N \right) \right) \times  M_N^{n-1} \overline{M_N}^{n} \right] \\
		&= - \frac{e^{-t}}{2} \sum_{1\leq s\leq d} \E\left[ \left( e^{t/2} \ts_N\left( D_sQ(X_t^N) Y_s^N\right) - \ts_N^{\otimes 2}\left( \partial_sD_sQ(X_t^N) \right)\right) \times  M_N^{n-1} \overline{M_N}^{n} \right],
	\end{align*}
	where we used the Schwinger-Dyson equations, i.e. Proposition \ref{skscksd}, to get the last line. Note that from \eqref{dfvkdnc}, the non-commutative differential of $X_t^N$ with respect to $X^N$ is $(1-e^{-t})^{1/2} 1\otimes 1$, hence why the factor $(1-e^{-t})^{-1/2}$ disappears from the second to the third line above. Since 
	$$ \E\left[ \ts_N\left( Q(X_0^N) \right)\times M_N^{n-1} \overline{M_N}^{n}\right] = \E\left[ \ts_N\left( Q(Y^N) \right)\times  M_N^{n-1} \overline{M_N}^{n} \right], $$
	$$ \lim_{t\to\infty} \E\left[ \ts_N\left( Q(X_t^N) \right)\times  M_N^{n-1} \overline{M_N}^{n} \right] = \E\left[ \ts_N\left( Q(X^N) \right)\right] \times \E\left[M_N^{n-1} \overline{M_N}^{n} \right], $$
	we have that
	\begin{align}
	\label{sdpdso}
		\E\left[ |M_N|^{2n} \right] &= \frac{1}{2} \sum_{1\leq s\leq d} \int_0^{\infty} e^{-t} \E\left[ \left( e^{t/2} \ts_N\left( D_sQ(X_t^N) Y_s^N\right) - \ts_N^{\otimes 2}\left( \partial_sD_sQ(X_t^N) \right)\right) \times  M_N^{n-1} \overline{M_N}^{n}\right] dt.
	\end{align}
	 Next we want to use the cumulant expansion in Proposition~\ref{cuulantmexpansion} to the third order ($\ell+1=3)$ with the choices
  \begin{equation}
  u=\Re((Y_s^N)_{i,j})\,\qquad v=\Im((Y_s^N)_{i,j})
  \end{equation}
  and
  \begin{equation}
      \Phi_{i,j}: \left(\Re((Y_s^N)_{i,j}),\Im((Y_s^N)_{i,j})\right) \longmapsto e^{t/2} \ts_N\left( D_sQ(X_t^N) E_{i,j} \right) M_N^{n-1} \overline{M_N}^{n}.
  \end{equation}
   Here the $\Phi_{i,j}$ is chosen to reproduce the first term in the integrand on the right side of~\eqref{sdpdso}, i.e.
    \begin{align*}
        \E\left[ \left( e^{t/2} \ts_N\left( D_sQ(X_t^N) Y_s^N\right)\right)M_N^{n-1}\overline{M_N}^n\right]=\sum_{1\le i,j\le N} \E\left[(Y_s^N)_{i,j} \Phi_{i,j}\left(\Re((Y_s^N)_{i,j}),\Im((Y_s^N)_{i,j})\right)\right].
    \end{align*}

    \noindent We now compute the first derivatives of $\Phi_{i,j}$. If $i\neq j$, we have that 
	\begin{align}
		\label{idxud}
		\partial_u \Phi_{i,j}(u,v) =\ &\ts_N\left( \big(\partial_s D_sQ(X_t^N)\# (E_{i,j}+E_{j,i})\big) E_{i,j} \right) M_N^{n-1} \overline{M_N}^{n} \nonumber \\
		&+ \left(n-1\right) \ts_N\left( D_sQ(X_t^N) E_{i,j} \right) \ts_N\left( D_s{Q}(X_0^N) (E_{i,j}+E_{j,i}) \right) M_N^{n-2} \overline{M_N}^{n} \\
		&+ n \ts_N\left( D_sQ(X_t^N) E_{i,j} \right) \ts_N\left( D_sQ^*(X_0^N) (E_{i,j}+E_{j,i}) \right) M_N^{n-1} \overline{M_N}^{n-1}, \nonumber
	\end{align}
	and
	\begin{align}
		\label{idxud2}
		\partial_v \Phi_{i,j}(u,v) =\ &\i \ts_N\left( \big(\partial_s D_sQ(X_t^N)\# (E_{i,j}-E_{j,i})\big) E_{i,j} \right) M_N^{n-1} \overline{M_N}^{n} \nonumber \\
		&+\i \left(n-1\right) \ts_N\left( D_sQ(X_t^N) E_{i,j} \right) \ts_N\left( D_s{Q}(X_0^N) (E_{i,j}-E_{j,i}) \right) M_N^{n-2} \overline{M_N}^{n} \\
		&+\i n \ts_N\left( D_sQ(X_t^N) E_{i,j} \right) \ts_N\left( D_sQ^*(X_0^N) (E_{i,j}-E_{j,i}) \right) M_N^{n-1} \overline{M_N}^{n-1}. \nonumber
	\end{align}
	Further if $i=j$, then $\partial_u \Phi_{i,j}(u,v)$ is obtained similarly but with $E_{i,i}$ instead of $E_{i,j}+E_{j,i}$, and $\partial_v \Phi_{i,j}(u,v)=0$. Besides with $u$ and $v$ defined as such, their cumulants satisfy
	\begin{align}
	\label{eqcumul}
		&\kappa_{1,0} = \frac{1}{2}\Re\E[(Y_s^N)_{i,j}]=0 , \nonumber \\ 
		&\kappa_{0,1} = \frac{1}{2\i}\Im\E[(Y_s^N)_{i,j}]=0 , \nonumber \\ 
		&\kappa_{2,0}+\kappa_{0,2} = \E[|(Y_s^N)_{i,j}|^2],\nonumber \\
		&\kappa_{2,0} - \kappa_{0,2} +2\i \kappa_{1,1} = \E[(Y_s^N)_{i,j}^2]. 
	\end{align}

	\noindent In particular, as we set in Definition \ref{defwigner}, if $i\neq j$, $\kappa_{2,0}+\kappa_{0,2} = 1/N$. Thus by using the cumulant expansions in Proposition \ref{cuulantmexpansion} with $\ell=2$, we get that
 \begin{subequations}
 \label{guewig full set}
	\begin{align}
		\label{guewig}
		 &e^{t/2}\E\left[ (Y_s^N)_{i,j} \times \ts_N\left( D_sQ(X_t^N) E_{i,j} \right) M_N^{n-1} \overline{M_N}^{n} \right] \\
		\label{guewig2}
		 &= \frac{1}{N}\E\left[ \ts_N\left( (\partial_s D_sQ(X_t^N)\# E_{j,i}) E_{i,j} \right) M_N^{n-1} \overline{M_N}^{n} \right] \\
		\label{guewig3} &\quad+ \frac{n-1}{N} \E\left[ \ts_N\left( D_sQ(X_t^N) E_{i,j} \right) \ts_N\left( D_s{Q}(X_0^N) E_{j,i} \right) M_N^{n-2} \overline{M_N}^{n} \right] \\ \label{guewig4}
		 &\quad+ \frac{n}{N} \E\left[ \ts_N\left( D_sQ(X_t^N) E_{i,j} \right) \ts_N\left( D_sQ^*(X_0^N) E_{j,i} \right) M_N^{n-1} \overline{M_N}^{n-1} \right] \\ \label{guewig5}
		 &\quad+ \left( \E[(Y_s^N)_{i,j}^2] -\frac{\1_{i=j}}{N} \right)\E\left[ \ts_N\left( (\partial_s D_sQ(X_t^N)\# E_{i,j}) E_{i,j} \right) M_N^{n-1} \overline{M_N}^{n} \right] \\ \label{guewig6}
		 &\quad+ \left(n-1\right)\left( \E[(Y_s^N)_{i,j}^2] -\frac{\1_{i=j}}{N} \right) \E\left[ \ts_N\left( D_sQ(X_t^N) E_{i,j} \right) \ts_N\left( D_s{Q}(X_0^N) E_{i,j} \right) M_N^{n-2} \overline{M_N}^{n} \right] \\ \label{guewig7}
		 &\quad+ n \left( \E[(Y_s^N)_{i,j}^2] -\frac{\1_{i=j}}{N} \right) \E\left[ \ts_N\left( D_sQ(X_t^N) E_{i,j} \right) \ts_N\left( D_sQ^*(X_0^N) E_{i,j} \right) M_N^{n-1} \overline{M_N}^{n-1} \right] \\ \label{guewig8}
		 &\quad+ \frac{1}{N^{3/2}} H_{i,j}^{N,n} + R^{N,n}_{i,j}, 
	\end{align}
 \end{subequations}
	where $H_{i,j}^{N,n}$ is the term coming from the second order derivatives (i.e. $\partial_u^2 \Phi_{i,j}(u,v)$, $\partial_v^2 \Phi_{i,j}(u,v)$ and $\linebreak \partial_u\partial_v \Phi_{i,j}(u,v)$), and $R^{N,n}_{i,j}$ is the remainder $\varepsilon_3$, consequently this term correspond to the third order derivatives of $\Phi_{i,j}$. \\
 
	\textbf{Step 2 (Bounding the first order error term):} In this step, we focus on upper bounding the terms in Equations \eqref{guewig3} to~\eqref{guewig7}, to do so we use the following inequality. Given two matrices $A$ and $B$ of size $N$,
	\begin{equation}
		\sum_{1\leq i,j\leq N} |A_{i,j}B_{i,j}| \leq \sqrt{\sum_{1\leq i,j\leq N} |A_{i,j}|^2} \sqrt{\sum_{1\leq i,j\leq N} |A_{i,j}|^2} = N \sqrt{\ts_{N}(AA^*)\ts_{N}(BB^*)}.
	\end{equation}
	Consequently, thanks to the moment assumption in~\eqref{assump}, after summing over $i$ and $j$, one can bound the terms in~\eqref{guewig3}, ~\eqref{guewig4}, \eqref{guewig6} and~\eqref{guewig7} by
	\begin{equation}
		\label{extraterm1}
		\frac{y^2}{N^2} \E\left[|M_N|^{2n-2}  \sqrt{\ts_N(B_1(Z^N)) \ts_N(B_2(Z^N))}\right], 
	\end{equation}
	for some $B_1,B_2\in\F_{2d,q}$ and $Z^N=(X^N,Y^N,A^N)$. Indeed, for example, if one looks separately at \eqref{guewig6},
    \begin{align*}
        &\sum_{1\leq i,j\leq N} \left| \left( \E[(Y_s^N)_{i,j}^2] -\frac{\1_{i=j}}{N} \right) \E\left[ \ts_N\left( D_sQ(X_t^N) E_{i,j} \right) \ts_N\left( D_s{Q}(X_0^N) E_{i,j} \right) M_N^{n-2} \overline{M_N}^{n} \right] \right| \\
        &\leq \frac{\sup_{i,j,N} N\E[|(Y_s^N)_{i,j}|^2]+1}{N^3} \E\left[ \sum_{1\leq i,j\leq N} \left| \left(D_sQ(X_t^N) \right)_{j,i}\right| \left|\left( D_s{Q}(X_0^N) \right)_{j,i}\right| |M_N|^{2n-2} \right]\\
        &\leq \frac{C_2}{N^2} \E\left[ \sqrt{\ts_N\left(D_sQ(X_t^N)^*D_sQ(X_t^N) \right)} \sqrt{\ts_N\left( D_s{Q}(X_0^N)^*D_s{Q}(X_0^N) \right)} |M_N|^{2n-2} \right],
    \end{align*}
    with $C_2 = \sup_{i,j,N} N\E[|(Y_s^N)_{i,j}|^2]+1$, which is finite thanks to moment assumption in~\eqref{assump}. Then if we write $D_sQ = \sum_l c_l M_l$ where $c_l\in\C$ and $M_l$ are unitary monomials in $X_t^N$ and $e^{\i R(X_t^N)}$ (for any self-adjoint $R$), then thanks to Equation \eqref{3ext}, $d_l=\sup_{y_1,\dots,y_k\in\R} c_l/y$ is finite and
    \begin{align*}
        &\sqrt{\ts_N\left(D_sQ(X_t^N)^*D_sQ(X_t^N) \right)} \\
        &\leq \sum_l |c_l| \sqrt{\ts_N\left(M_l(X_t^N)^*M_lQ(X_t^N) \right)} \\
        &\leq y \sum_l |d_l| \sqrt{\ts_N\left(M_l(X_t^N)^*M_lQ(X_t^N) \right)} \\
        &\leq y K \sqrt{\sum_l \ts_N\left(M_l(X_t^N)^*M_lQ(X_t^N) \right)},
    \end{align*}
    for some constant $K$. Hence we can find  $B_1,B_2\in\F_{2d,q}$ such that
    \begin{align*}
        &\sum_{1\leq i,j\leq N} \left| \left( \E[(Y_s^N)_{i,j}^2] -\frac{\1_{i=j}}{N} \right) \E\left[ \ts_N\left( D_sQ(X_t^N) E_{i,j} \right) \ts_N\left( D_s{Q}(X_0^N) E_{i,j} \right) M_N^{n-2} \overline{M_N}^{n} \right] \right| \\
        &\leq \frac{y^2}{N^2} \E\left[|M_N|^{2n-2}  \sqrt{\ts_N(B_1(Z^N)) \ts_N(B_2(Z^N))}\right].
    \end{align*}
    
    \noindent Similarly after summing over $i$ and $j$, one can bound the terms in~\eqref{guewig5} by 
    \begin{equation}
		\label{extraterm2}
		\frac{y^2}{N} \E\left[|M_N|^{2n-1}  \sqrt{\ts_N(B_1(Z^N)) \ts_N(B_2(Z^N))}\right], 
	\end{equation}
    Then by using Lemma \ref{podvfldmvd} with $\alpha=y^2/N^{1-\varepsilon}$, $\mathcal{X}=N^{-\frac{g}{2n-g}\varepsilon} M_N$, $\mathcal{Y}=\sqrt{\ts_N(B_1(Z^N)) \ts_N(B_2(Z^N))}$, and $g=1$ or $2$, we get that \eqref{extraterm1} and \eqref{extraterm2} can be upper bounded by
    $$ N^{-\varepsilon} \E\left[|M_N|^{2n}\right] + \left(\frac{y^2}{N^{1-\varepsilon}}\right)^{2n} \E\left[ \left(\ts_N(B_1(Z^N)) \ts_N(B_2(Z^N))\right)^{\frac{n}{g}} \right]. $$
    Thus, thanks to Lemma \ref{docsocdssdc}, there exists a constant $C_n$ such that the terms in Equations \eqref{extraterm1} and \eqref{extraterm2} can be upper bounded by
    \begin{equation*}
        N^{-\varepsilon} \E\left[|M_N|^{2n}\right] + C_n \left(\frac{y^2}{N^{1-\varepsilon}}\right)^{2n}.
    \end{equation*}
    Consequently, after summing over $i,j$, Equation \eqref{guewig full set} yields
	\begin{align}
		\label{guewig11}
        &\left| \E\left[ \left( e^{t/2} \ts_N\left( D_sQ(X_t^N) Y_s^N\right) - \ts_N^{\otimes 2}\left( \partial_sD_sQ(X_t^N) \right)\right) \times  M_N^{n-1} \overline{M_N}^{n} \right] \right|\\
		&\leq \frac{1}{N^{3/2}} \sum_{1\leq i,j\leq N} H_{i,j}^{N,n} + \sum_{1\leq i,j\leq N} R^{N,n}_{i,j} + N^{-\varepsilon} \E\left[|M_N|^{2n}\right] + C_n \left(\frac{y^2}{N^{1-\varepsilon}}\right)^{2n}.  \nonumber
	\end{align}

	\textbf{Step 3  (Bounding the second order error term):} Next we want to tackle the term $H_{i,j}^{N,n}$ in~\eqref{guewig11}. We recall that it is the term coming from the second order derivatives, i.e. $\partial_u^2 \Phi_{i,j}(u,v)$, $\partial_v^2 \Phi_{i,j}(u,v)$ and $\partial_u\partial_v \Phi_{i,j}(u,v)$, in \eqref{guewig8}. The first order derivatives of $\Phi_{i,j}$ are given in Equations \eqref{idxud} and \eqref{idxud2}. Besides, we compute
	\begin{align*}
		\partial_u M_N^{n} =\ e^{-t/2}n \ts_N\left(  D_sQ(X_t^N) (E_{i,j}+E_{j,i}) \right) M_N^{n-1},
	\end{align*}
    \begin{align*}
		\partial_v M_N^{n} =\ \i e^{-t/2}n \ts_N\left(  D_sQ(X_t^N) (E_{i,j}-E_{j,i}) \right) M_N^{n-1},
	\end{align*}
    \begin{align*}
		\partial_u \ts_N\left( D_sQ(X_t^N) E_{i,j} \right) =\ e^{-t/2}\ts_N\left( \big(\partial_s D_sQ(X_t^N)\# (E_{i,j}+E_{j,i})\big) E_{i,j} \right) ,
    \end{align*}
    \begin{align*}
		\partial_v \ts_N\left( D_sQ(X_t^N) E_{i,j} \right) =\ \i e^{-t/2}\ts_N\left( \big(\partial_s D_sQ(X_t^N)\# (E_{i,j}-E_{j,i})\big) E_{i,j} \right) ,
    \end{align*}
    \begin{align*}
		&\partial_u \ts_N\left( \big(\partial_s D_sQ(X_t^N)\# (E_{i,j}+E_{j,i})\big) E_{i,j} \right) \\
        &=  e^{-t/2} \ts_N\left( \big((\partial_s\otimes\id)\circ\partial_s D_sQ(X_t^N)\# (E_{i,j}+E_{j,i},E_{i,j}+E_{j,i})\big) E_{i,j} \right) \\
        &\quad + e^{-t/2} \ts_N\left( \big((\id\otimes \partial_s)\circ\partial_s D_sQ(X_t^N)\# (E_{i,j}+E_{j,i},E_{i,j}+E_{j,i})\big) E_{i,j} \right),
	\end{align*}
    \begin{align*}
		&\partial_v \ts_N\left( \big(\partial_s D_sQ(X_t^N)\# (E_{i,j}+E_{j,i})\big) E_{i,j} \right) \\
        &= \i e^{-t/2} \ts_N\left( \big((\partial_s\otimes\id)\circ\partial_s D_sQ(X_t^N)\# (E_{i,j}-E_{j,i},E_{i,j}+E_{j,i})\big) E_{i,j} \right) \\
        &\quad + \i e^{-t/2} \ts_N\left( \big((\id\otimes \partial_s)\circ\partial_s D_sQ(X_t^N)\# (E_{i,j}+E_{j,i},E_{i,j}-E_{j,i})\big) E_{i,j} \right).
	\end{align*}
    Moreover, thanks to the moment assumption in~\eqref{assump}, one can bound the mixed cumulants of the real and imaginary part of $\sqrt{N}(Y_s^N)_{i,j}$ uniformly over $i,j$ and $N$. Note that since this term comes from the second order derivative, the cumulants which appears due to Proposition \ref{cuulantmexpansion} are of order $3$, i.e. $\kappa_{n,m}$ with $n+m=3$, hence why we normalize $H_{i,j}^{N,n}$ by $N^{3/2}$ in Equation \eqref{guewig8}. Consequently, $H_{i,j}^{N,n}$ can be bounded by a linear combination (whose coefficients are independent of $i,j$) of terms of the form
	$$ \E\left[ |\ts_N(AF_1BF_2CF_3)|\times  |M_N|^{2n-1} \right] ,$$
	$$ n\E\left[| \ts_N(AF_1BF_2) \ts_{N}(C F_3)|\times |M_N|^{2n-2} \right] ,$$
	$$ n^2\E\left[| \ts_N(AF_1) \ts_{N}(B F_2) \ts_{N}(C F_3)|\times |M_N|^{2n-3} \right] ,$$
	
	\noindent with $F_1,F_2,F_3\in \{E_{i,j},E_{j,i}\}$ and $A,B,C$ elements of $\F_{d,q}$ evaluated in $X^N_t$. So one can sum $H_{i,j}^{N,n}$ over $i,j$ and by using the following kind of inequalities,
	\begin{align}
		\sum_{1\leq i,j\leq N} |A_{i,j} B_{i,j} C_{i,j}| &\leq \sqrt{\sum_{1\leq i,j\leq N} |A_{i,j} B_{i,j}|^2} \sqrt{\sum_{1\leq i,j\leq N} |C_{i,j}|^2}  \\
		&\leq \sqrt{\sum_{1\leq i,j\leq N} |A_{i,j}|^2 \sum_{1\leq i,j\leq N} |B_{i,j}|^2 \tr_N(C^*C) } \nonumber \\
		&\leq \sqrt{\tr_N(A^*A) \tr_N(B^*B) \tr_N(C^*C) } \nonumber \\
		&\leq N^{3/2} \sqrt{\ts_N(A^*A) \ts_N(B^*B) \ts_N(C^*C) }, \nonumber
	\end{align}
	and
	\begin{align}
		\label{optieked}
		\sum_{1\leq i,j\leq N} |A_{i,i} B_{j,j} C_{i,j}| &\leq \sqrt{\sum_{1\leq i,j\leq N} |A_{i,i} B_{j,j}|^2} \sqrt{\sum_{1\leq i,j\leq N} |C_{i,j}|^2} \\
		&\leq \sqrt{\tr_N(A^*A) \tr_N(B^*B) \tr_N(C^*C) } \nonumber \\
		&\leq N^{3/2} \sqrt{\ts_N(A^*A) \ts_N(B^*B) \ts_N(C^*C) }, \nonumber
	\end{align}	

	\noindent one has the following bounds for some constant $C_n$, and $B_1,B_2,B_3\in\F_{2d,q}$ and $Z^N=(X^N,Y^N,A^N)$,
    \begin{align*}
         \frac{1}{N^{3/2}}  \sum_{i,j} | H_{i,j}^{N,n}| \leq C_n \E\Bigg[ &\left( \frac{y^3}{N} |M_N|^{2n-1} + \frac{y^3}{N^2} |M_N|^{2n-2} + \frac{y^3}{N^3} |M_N|^{2n-3} \right) \\
         &\quad\times \sqrt{\ts_N(B_1(Z^N)) \ts_N(B_2(Z^N)) \ts_N(B_3(Z^N))} \Bigg].  
    \end{align*}
    Next, we use Lemma \ref{podvfldmvd} with 
    $$\alpha=y^3/N^{1-\varepsilon},\ \mathcal{X}=N^{-\frac{g}{2n-g}\varepsilon} M_N,\ \mathcal{Y}=C_n\sqrt{\ts_N(B_1(Z^N)) \ts_N(B_2(Z^N)) \ts_N(B_3(Z^N))},$$
    and $g=1,2,3$. Thanks to Lemma \ref{docsocdssdc}, we get that there exists a constant $C_{n}$ such that
    \begin{equation}
        \label{soimcoskmocsmc}
        \frac{1}{N^{3/2}}\sum_{i,j} | H_{i,j}^{N,n}| \leq N^{-\varepsilon} \E\left[ |M_N|^{2n} \right] + C_{n}\left(\frac{y^3}{N^{1-\varepsilon}}\right)^{2n}.
    \end{equation}
	
	\textbf{Step 4  (Bounding the third order error term):} It remains to bound $R^{N,n}_{i,j}$ in~\eqref{guewig11}. Since we have an explicit expression for the remainder $\varepsilon_3$ in Proposition \ref{cuulantmexpansion}, we let $ M_N^{u,i,j,s}$ be defined just like $M_N$ but with the entries $(j,i)$ and $(i,j)$ of the matrix $Y_s^N$ multiplied by $u\in[0,1]$. Then one has that $R^{N,n}_{i,j}$ is bounded by a linear combination (whose coefficients are independent of $i,j$) of terms of the form,
	\begin{equation}
		\label{kjfd}
		\frac{y^4}{N^{2-r/2}} \E\left[ \left| (Y_s^N)_{i,j} \right|^{r} \int_{0}^{1} \frac{|M_N^{u,i,j,s}|^{2n-m}}{N^m} \left| \ts_N(B_1(\widetilde{Z}^N))\dots \ts_N(B_l(\widetilde{Z}^N))\right|^{1/8} du \right] ,
	\end{equation}
	where $r,m\in [1,4 ] $, $B_1,\dots,B_l\in\F_{d,q}$, and $\widetilde{Z}^N$ is defined like $Z^N$ but with the entries $(j,i)$ and $(i,j)$ of the matrix $Y_s^N$ multiplied by $u$. This is notably due to the fact that by Hölder's inequality, for example, for any matrices $R,S,T,U$,
    $$ \ts_N(R E_{i,j} S E_{i,j} T E_{j,i} U E_{i,j}) \leq \ts_N(|R|^8)^{1/8}\ts_N(|S|^8)^{1/8} \ts_N(|T|^8)^{1/8} \ts_N(|U|^8)^{1/8}.$$
    Besides, given a matrix $R$, 
    \begin{equation}
    	\label{epsi}
        \left|\ts_N\left(R\left(E_{i,j}(Y_s^N)_{i,j}+E_{j,i}(Y_s^N)_{j,i}\right)\right)\right| \leq 2\sqrt{\ts_N(R^*R)} \frac{\left| (Y_s^N)_{j,i} \right|}{\sqrt{N}},
    \end{equation}
    thus we have that for any $u\in [0,1]$,
	\begin{equation}
		|M_N^{u,i,j,s} - M_N| \leq \frac{y \left| (Y_s^N)_{j,i} \right|}{\sqrt{N}} P\left(\left| (Y_s^N)_{j,i} \right|\right) \left|\ts_N(C_1(Z^N))\dots \ts_N(C_{l'}(Z^N))\right|^{1/c},
	\end{equation}
	for some polynomials $P,C_1,\dots,C_{l'}$ and a constant $c$. Hence, one can bound \eqref{kjfd} by a linear combination (whose coefficients are once again independent of $i,j$) of terms of the form, for $h\in[0,2n-m]$,
	\begin{align*}
		\frac{y^{4+h}}{N^{2+(h-r)/2+m}} \E\Bigg[ \left| (Y_s^N)_{i,j} \right|^{r+h} |M_N|^{2n-m-h} & \left| \ts_N(B_1(\widetilde{Z}^N))\dots \ts_N(B_l(\widetilde{Z}^N))\right|^{1/8} \\
        &\times P\left(\left| (Y_s^N)_{j,i} \right|\right) \left|\ts_N(C_1(Z^N))\dots \ts_N(C_{l'}(Z^N))\right|^{1/c} \Bigg] .
	\end{align*}
    Next, by using Equation \eqref{epsi} again, one can find polynomials $P,D_1,\dots,D_{l''}$ and a constant $c'$ such that the quantity above is bounded by
    \begin{align*}
		\frac{y^{4+h}}{N^{2+h+m}} \E\Bigg[ \left| \sqrt{N} (Y_s^N)_{i,j} \right|^{r+h} |M_N|^{2n-m-h} & Q\left(\left| (Y_s^N)_{j,i} \right|\right) \left|\ts_N(D_1(Z^N))\dots \ts_N(D_{l''}(Z^N))\right|^{1/c'} \Bigg].
	\end{align*}
    We can then use Lemma \ref{podvfldmvd} with $$\alpha=y^4/N^{1-\varepsilon},\  g=m+h,\ \mathcal{X}=N^{-\frac{g}{2n-g}\varepsilon} |M_N|, $$
    $$ \mathcal{Y}=\left| \sqrt{N} (Y_s^N)_{i,j} \right|^{r+h} Q\left(\left| (Y_s^N)_{j,i} \right|\right) \left|\ts_N(D_1(Z^N))\dots \ts_N(D_{l''}(Z^N))\right|^{1/c'}.$$
    Then thanks to the moment assumption in~\eqref{assump} and Lemma \ref{docsocdssdc}, we can once again upper bound Equation \eqref{kjfd} by
    \begin{equation*}
        \frac{1}{N^2}\left(N^{-\varepsilon} \E\left[ |M_N|^{2n} \right] + C_{n}\left(\frac{y^4}{N^{1-\varepsilon}}\right)^{2n}\right)
    \end{equation*}
    for some constant $C_n$. Consequently, we have for some other constant $C_n$, that
    \begin{align}
		\label{guewig1lscd1}
        \sum_{1\leq i,j\leq N} R^{N,n}_{i,j} \leq C_n\left( N^{-\varepsilon} \E\left[|M_N|^{2n}\right] + \left(\frac{y^4}{N^{1-\varepsilon}}\right)^{2n} \right).
	\end{align}
    Thus by combining Equations \eqref{guewig11}, \eqref{soimcoskmocsmc} and \eqref{guewig1lscd1}, we get that for some constant $C_n$
    \begin{align}
		\label{guewig1domcs1}
        &\left| \E\left[ \left( e^{t/2} \ts_N\left( D_sQ(X_t^N) Y_s^N\right) - \ts_N^{\otimes 2}\left( \partial_sD_sQ(X_t^N) \right)\right) \times  M_N^{n-1} \overline{M_N}^{n} \right] \right|\\
		&\leq C_n\left( N^{-\varepsilon} \E\left[|M_N|^{2n}\right] + \left(\frac{y^4}{N^{1-\varepsilon}}\right)^{2n} \right).  \nonumber
	\end{align}
    Hence, by plugging this result in Equation \eqref{sdpdso}, for $N$ sufficiently large, 
    \begin{align*}
        \E\left[ |M_N|^{2n} \right] \leq \frac{C_n}{1-\frac{d}{2}C_n N^{-\varepsilon}} \left(\frac{y^4}{N^{1-\varepsilon}}\right)^{2n} \leq 2C_n \left(\frac{y^4}{N^{1-\varepsilon}}\right)^{2n}.
	\end{align*}
    With the above moment estimate at hand, we can apply Markov's inequality: for any $\delta>0$, we have for any $n$, for $N$ large enough,
	\begin{align*}
		&\P\left( \left| \ts_N\left(Q(Y^N,A^N)\right) - \E[\ts_N\left(Q(X^N,A^N)\right)] \right| \geq \delta \right) \leq \E[|M_N|^{2n}] \delta^{-2n} \leq 2C_{n} \left(\frac{y^4}{N^{1-\varepsilon}}\right)^{2n} \delta^{-2n}.
	\end{align*}
    Hence, we choose $\delta= y^4N^{2\varepsilon-1}$, and we have that for any $n$, for $N$ large enough,
	\begin{align*}
		&\P\left( \left| \ts_N\left(Q(Y^N,A^N)\right) - \E[\ts_N\left(Q(X^N,A^N)\right)] \right| \geq \left(\frac{y^4}{N^{1-2\varepsilon}} \right)\right) \leq 2C_{n} N^{-2\varepsilon n}.
	\end{align*}
	Consequently, for any $\varepsilon>0$, we have with high probability that
	$$ \ts_N\left(Q(Y^N,A^N)\right) = \E[\ts_N\left(Q(X^N,A^N)\right)] + \mathcal{O}\left(\frac{y^{4}}{N^{1-\varepsilon}}\right),$$
	which proves Equation $\eqref{lkeshfs}$. \\
	
	\textbf{Step 5 (The case of the GUE and the GOE):} In the case where $Y_s^N$ is a GOE random matrix, then by using Proposition \ref{3SD} (with a slight modification to take into account for the term $ M_N^{n-1} \overline{M_N}^{n}$) one has that
	\begin{align}
		\label{oidsoes}
		&e^{t/2}\E\left[\ts_N\left( D_sQ(X_t^N) Y_s^N\right) M_N^{n-1} \overline{M_N}^{n} \right] \nonumber \\
		&= \frac{1}{N}\E\left[ \ts_N^{\otimes 2}\left( \partial_s D_sQ(X_t^N) \right) M_N^{n-1} \overline{M_N}^{n} \right] \nonumber\\
		&\quad+ \frac{1}{N^2} \left(n-1\right) \E\left[ \ts_N\left( D_sQ(X_t^N) D_s{Q}(X_0^N) \right) M_N^{n-2} \overline{M_N}^{n} \right] \nonumber\\
		&\quad+ \frac{n}{2N^2} \E\left[ \ts_N\left( D_sQ(X_t^N) \left(D_sQ(X_0^N)\right)^* \right) M_N^{n-1} \overline{M_N}^{n-1} \right] \\ \linebreak
		&\quad+\frac{1}{N}\E\left[ \ts_N\left( h\left(\partial_s D_sQ(X_t^N)\right) \right) M_N^{n-1} \overline{M_N}^{n} \right] \nonumber\\
		&\quad+ \frac{1}{N^2} \left(\frac{n}{2}-1\right) \E\left[ \ts_N\left(\left( D_sQ(X_t^N)\right)^T D_s{Q}(X_0^N) \right) M_N^{n-2} \overline{M_N}^{n} \right] \nonumber\\
		&\quad+ \frac{n}{2N^2} \E\left[ \ts_N\left( \left(D_sQ(X_t^N)\right)^T \left(D_sQ(X_0^N)\right)^* \right) M_N^{n-1} \overline{M_N}^{n-1} \right], \nonumber
	\end{align}
	where $A^T$ is the transpose of $A$, and $h$ is the linear map such that $h(A\otimes B) = A^TB$. Thus, similarly to Equation \eqref{guewig11}, we get that 
    \begin{align*}
        &\left| \E\left[ \left( e^{t/2} \ts_N\left( D_sQ(X_t^N) Y_s^N\right) - \ts_N^{\otimes 2}\left( \partial_sD_sQ(X_t^N) \right)\right) \times M_N^{n-1} \overline{M_N}^{n} \right] \right|\\
		&\leq N^{-\varepsilon} \E\left[|M_N|^{2n}\right] + C_n \left(\frac{y^2}{N^{1-\varepsilon}}\right)^{2n}. 
	\end{align*}
	Hence by reinjecting this result in Equation \eqref{sdpdso}, we prove that $\E[|M_N|^n] = \mathcal{O}((y^{2}N^{\varepsilon-1})^{2n})$ for any $n$. 
 
 The case where $Y_s^N$ is a GUE random matrix is even simpler. Indeed, we compute that
	\begin{align*}
		&e^{t/2}\E\left[\ts_N\left( D_sQ(X_t^N) Y_s^N\right) M_N^{n-1} \overline{M_N}^{n} \right] \nonumber \\
		&= \frac{1}{N}\E\left[ \ts_N^{\otimes 2}\left( \partial_s D_sQ(X_t^N) \right) M_N^{n-1} \overline{M_N}^{n} \right] \nonumber\\
		&\quad+ \frac{1}{N^2} \left(\frac{n}{2}-1\right) \E\left[ \ts_N\left( D_sQ(X_t^N) D_s{Q}(X_0^N) \right) M_N^{n-2} \overline{M_N}^{n} \right] \nonumber\\
		&\quad+ \frac{n}{2N^2} \E\left[ \ts_N\left( D_sQ(X_t^N) \left(D_sQ(X_0^N)\right)^* \right) M_N^{n-1} \overline{M_N}^{n-1} \right],
	\end{align*}
	and the result follows.
\end{proof}

\begin{rem}
\label{sckmsssscd}
	We suspect that it should be possible to improve the error term $N^\epsilon \frac{y^4}{N}$ in~\eqref{lkeshfs} for Wigner matrices to $N^\epsilon \frac{y^2}{N}$, as is the case for GOE and GUE matrices. The weaker bound in the general case comes from certain inequalities used to bound contributions from the second and third order terms in the cumulant expansions that cannot be easily improved for general polynomials, e.g.\eqref{optieked} or \eqref{kjfd}, as one does not expect cancellations from higher order terms in the cumulant expansion.
\end{rem}

\subsection{Concentration of the scalar product}

Next we prove a similar concentration result for the scalar product. The main difference is that instead of having a concentration of order $N^{-1}$, it is only of order $N^{-1/2}$. Note that this speed of convergence cannot be improved by pushing the cumulant expansion further since it comes from the first error term (see Equation \eqref{guewigner3} to \eqref{guewigner7}). Besides it is in line with results obtained by proving local, see for example Theorem 2.6 of \cite{benanti}. The main difference between local laws and Proposition \ref{concentrcoeff} comes from the term $y^4$ which as we discussed in Remark \ref{sckmsssscd} we suspect is not optimal. Before giving the concentration estimate we prove the following lemma which we will need.

\begin{lemma}
\label{docsocscddssdc}
Given
\begin{itemize}
		\item $Y^N$ a $d$-tuple of independent Wigner matrices as in Definition \ref{defwigner}, 
		\item $A^N$ deterministic matrices such that $\sup_{1\leq i\leq q,N\in\N^*} \norm{A_i^N} <\infty$,
        \item $Q_1,\dots,Q_l\in\F_{d,q}$,
        \item $c\in\R^+$,
\end{itemize}
then for all $\gamma>0$,
\begin{equation}
    \sup_{N\in\N^*} \frac{1}{N^{\gamma}} \E\left[ \left( \norm{Q_1(Y^N,A^N)} \dots\norm{Q_l(Y^N,A^N)}\right)^c \right] <\infty.
\end{equation}
Besides if one writes each $Q_j$ as a linear combination of terms of the form $e^{\i P_1}R_1\cdots e^{\i P_k}R_k$, then the upper bound on the equation above does not depend on the (self-adjoint) polynomials $P_1,\dots,P_k$.
\end{lemma}

\begin{proof}
For any $k$ and $j$, one has that
\begin{align*}
    \norm{Q_j(Y^N,A^N)} &= \norm{\left(Q_j(Y^N,A^N)^*Q_j(Y^N,A^N)\right)^k}^{\frac{1}{2k}} \\
    &\leq \left(\tr_N\left(\left(Q_j(Y^N,A^N)^*Q_j(Y^N,A^N)\right)^k\right)\right)^{\frac{1}{2k}} \\
    &= N^{\frac{1}{2k}} \left(\ts_N\left(\left(Q_j(Y^N,A^N)^*Q_j(Y^N,A^N)\right)^k\right)\right)^{\frac{1}{2k}}.
\end{align*}
Thus thanks to Lemma \ref{docsocdssdc}, one has that for any $k$, there exists a constant $C_k$ such that for any $N$,
$$ \E\left[ \left( \norm{Q_1(Y^N,A^N)} \dots\norm{Q_l(Y^N,A^N)}\right)^c \right] \leq C_k N^{\frac{l}{2k}}.$$
Hence the conclusion by picking $k$ such that $\frac{l}{2k}\leq \gamma$.
\end{proof}

We can now state our concentration estimate.

\begin{prop}
	
	\label{concentrcoeff}
	
	Given 
	\begin{itemize}
		\item $Y^N$ a $d$-tuple of independent Wigner matrices as in Definition \ref{defwigner}, 
		\item $X^N$ a $d$-tuple of independent GUE random matrices, independent from $Y^N$,
		\item $A^N$ deterministic matrices such that $\sup_{1\leq i\leq q,N\in\N^*} \norm{A_i^N} <\infty$.
	\end{itemize}
	Let $\x,\y\in\C^N$, $y_i\in\R$, $y = 1 + \max_i |y_i|$, $P_1,\dots,P_k,R_1,\dots,R_k \in\A_{d,q}$ non-commutative polynomials. If we assume that $P_1,\dots,P_k$ are self-adjoint, then with  $Q= e^{\i P_1 y_1}R_1\cdots e^{\i P_k y_k}R_k$, we have for any $\varepsilon>0$ that with high probability, 
	\begin{align}\label{scalar concentration estimate}
		&\left\langle \x, Q(Y^N,A^N) \y \right\rangle = \E\left[ \left\langle \x, Q(X^N,A^N) \y \right\rangle \right] + \mathcal O\left( N^{\varepsilon} \frac{y^4}{\sqrt{N}} \norm{\x}_2\norm{\y}_2\right) .
	\end{align}
	Besides if $Y^N$ a $d$-tuple of independent GOE or GUE matrices, then one can replace $y^4$ by $y^2$ in the previous equality.
	
\end{prop}

Similarly to the proof of Lemma \ref{concentrtrace}, we divide the following proof it in four steps. The first one establishes an equation on the moment of $M_N$ (defined in Equation \eqref{slkmclskcs}). Then in each of the following two steps we bound a specific error term. In the last step, we study the case of the GUE and the GOE.

\begin{proof}[Proof of Proposition~\ref{concentrcoeff}]
	\textbf{Step 1 (Using the cumulant expansion):} First one has that for any matrix $R\in\M_N(\C)$, $\langle \x,R\y\rangle = \tr_N(R \y\x^*)$. Consequently, with $P=\y\x^*$, as previously we set
    \begin{equation}
        \label{slkmclskcs}
        M_N := \tr_{N}\Big(\big(Q(Y^N,A^N) - \E[Q(X^N,A^N)]\big)P\Big),
    \end{equation}
	and we want to study the moments of $M_N$ in order to use Markov's inequality. We want to prove that for any $n\in\N,\varepsilon>0$ there exists a constant $C_{n}$ such that for $N$ large enough
	\begin{equation}
		\label{butdelaproof}
		\E[|M_N|^{2n}] \leq C_{n} \left(\frac{y^4}{N^{1/2-\varepsilon}} \norm{P}\right)^{2n}.
	\end{equation}
	One can always assume that $y^4\leq \sqrt{N}$, otherwise, since 
	$$|M_N|\leq \left(\norm{Q(Y^N,A^N)} + \E\left[\norm{Q(X^N,A^N)}\right]\right)\norm{P},$$
	we immediately get Equation \eqref{butdelaproof} thanks Lemma \ref{docsocscddssdc}. First just like in Equation \eqref{sdpdso}, we have that for $n\in \N^*$,
	\begin{align}
		\label{sdpdso2}
		\E\left[ |M_N|^{2n} \right] = \frac{1}{2} \sum_{1\leq s\leq d} \int_0^{\infty} e^{-t} \E\Big[ \Big(& e^{t/2} \tr_N\left( \partial_sQ(X_t^N)\widetilde{\#}P \times Y_s^N\right) \\
		&- N^{-1}\tr_N^{\otimes 2}\left( \partial_s\big(\partial_sQ(X_t^N)\widetilde{\#}P\big) \Big)\Big) \times  M_N^{n-1} \overline{M_N}^{n} \right] dt. \nonumber
	\end{align}
	Then we want to use the cumulant expansion with $u=\Re((Y_s^N)_{i,j})$ and $v=\Im((Y_s^N)_{i,j})$, and 
	$$\Phi_{i,j}: \left(\Re((Y_s^N)_{i,j}),\Im((Y_s^N)_{i,j})\right) \longmapsto e^{t/2} \tr_N\left( \partial_sQ(X_t^N)\widetilde{\#}P \times E_{i,j}\right)  M_N^{n-1} \overline{M_N}^{n}. $$ 
	In particular if $i\neq j$, we have for the first derivatives of $\Phi$ that
	\begin{align}
		\label{idxud3}
		\partial_u \Phi_{i,j}(u,v) =\ &\tr_N\left( \left(\partial_s \left(\partial_sQ(X_t^N)\widetilde{\#}P\right)\# (E_{i,j}+E_{j,i})\right) E_{i,j} \right) M_N^{n-1} \overline{M_N}^{n} \\
		&+ \left(\frac{n}{2}-1\right) \tr_N\left( \left(\partial_sQ(X_t^N)\widetilde{\#}P\right) E_{i,j} \right) \ts_N\left( \left(\partial_sQ(X_0^N)\widetilde{\#}P\right) (E_{i,j}+E_{j,i}) \right) M_N^{n-2} \overline{M_N}^{n} \nonumber\\
		&+ \frac{n}{2} \tr_N\left( \left(\partial_sQ(X_t^N)\widetilde{\#}P\right) E_{i,j} \right) \ts_N\left( \left(\partial_sQ(X_0^N)\widetilde{\#}P\right)^* (E_{i,j}+E_{j,i}) \right) M_N^{n-1} \overline{M_N}^{n-1}, \nonumber
	\end{align}
	and
	\begin{align}
		\label{idxud4}
		\partial_v \Phi_{i,j}(u,v) =\ &\i \tr_N\left( \left(\partial_s \left(\partial_sQ(X_t^N)\widetilde{\#}P\right)\# (E_{i,j}-E_{j,i})\right) E_{i,j} \right) M_N^{n-1} \overline{M_N}^{n} \\
		&+\i \left(\frac{n}{2}-1\right) \tr_N\left( \left(\partial_sQ(X_t^N)\widetilde{\#}P\right) E_{i,j} \right) \ts_N\left( \left(\partial_sQ(X_0^N)\widetilde{\#}P\right) (E_{i,j}-E_{j,i}) \right) M_N^{n-2} \overline{M_N}^{n} \nonumber \\
		&+\i \frac{n}{2} \tr_N\left( \left(\partial_sQ(X_t^N)\widetilde{\#}P\right) E_{i,j} \right) \ts_N\left( \left(\partial_sQ(X_0^N)\widetilde{\#}P\right)^* (E_{i,j}-E_{j,i}) \right) M_N^{n-1} \overline{M_N}^{n-1}. \nonumber
	\end{align}
	Moreover if $i=j$, then $\partial_u \Phi_{i,j}(u,v)$ is defined similarly but with $E_{i,i}$ instead of $E_{i,j}+E_{j,i}$, and $\partial_v \Phi_{i,j}(u,v)=0$. By using again Equations \eqref{eqcumul} and Proposition \ref{cuulantmexpansion} with $\ell=1$, we get that
\begin{subequations}	\label{guewigner full set}
 \begin{align}
		&e^{t/2}\E\left[ (Y_s^N)_{i,j} \times \tr_N\left( \partial_s(Q(X_t^N)\widetilde{\#}P) E_{i,j} \right) M_N^{n-1} \overline{M_N}^{n} \right] 
		\\ \label{guewigner2}
		&= \frac{1}{N}\E\left[ \tr_N\left( \left(\partial_s \left(\partial_sQ(X_t^N)\widetilde{\#}P\right)\# E_{j,i}\right) E_{i,j} \right) M_N^{n-1} \overline{M_N}^{n} \right] 
	\\	\label{guewigner3}
		&\quad+ \frac{1}{N} \left(\frac{n}{2}-1\right) \E\left[ \tr_N\left( (\partial_sQ(X_t^N)\widetilde{\#}P) E_{i,j} \right) \tr_N\left( (\partial_sQ(X_0^N)\widetilde{\#}P) E_{j,i} \right) M_N^{n-2} \overline{M_N}^{n} \right] 
		\\ \label{guewigner4}
		&\quad+ \frac{n}{2N} \E\left[ \tr_N\left( (\partial_sQ(X_t^N)\widetilde{\#}P) E_{i,j} \right) \tr_N\left( (\partial_sQ^*(X_0^N)\widetilde{\#}P) E_{j,i} \right) M_N^{n-1} \overline{M_N}^{n-1} \right] \\ \label{guewigner5}
		&\quad+ \left( \E[(Y_s^N)_{i,j}^2] -\frac{\1_{i=j}}{N} \right) \E\left[ \ts_N\left( (\partial_s (\partial_sQ(X_t^N)\widetilde{\#}P)\# E_{i,j}) E_{i,j} \right) M_N^{n-1} \overline{M_N}^{n} \right] 
		\\ \label{guewigner6}
		&\quad+ \left(\frac{n}{2}-1\right)\left( \E[(Y_s^N)_{i,j}^2] -\frac{\1_{i=j}}{N} \right) \E\left[ \ts_N\left( (\partial_sQ(X_t^N)\widetilde{\#}P) E_{i,j} \right) \ts_N\left( (\partial_sQ(X_0^N)\widetilde{\#}P) E_{i,j} \right) M_N^{n-2} \overline{M_N}^{n} \right] 
		\\ \label{guewigner7}
		&\quad+ \frac{n}{2} \left( \E[(Y_s^N)_{i,j}^2] -\frac{\1_{i=j}}{N} \right) \E\left[ \ts_N\left( (\partial_sQ(X_t^N)\widetilde{\#}P) E_{i,j} \right) \ts_N\left( (\partial_sQ(X_0^N)\widetilde{\#}P)^* E_{i,j} \right) M_N^{n-1} \overline{M_N}^{n-1} \right]
		 \\ \nonumber
		&\quad+ R^{N,n}_{i,j}, 
	\end{align}
 \end{subequations}
    where $R^{N,n}_{i,j}$ is the remainder $\varepsilon_2$, consequently this term correspond to the second order derivatives. \\
    
    \textbf{Step 2 (Bounding the first order error term):} Next one can use the following inequality. Given two matrices $A$ and $B$ of size $N$ and rank~$1$,
	\begin{equation}
		\sum_{1\leq i,j\leq N} |A_{i,j}B_{i,j}| \leq \sqrt{\sum_{1\leq i,j\leq N} |A_{i,j}|^2} \sqrt{\sum_{1\leq i,j\leq N} |B_{i,j}|^2} = \sqrt{\tr_{N}(AA^*)\tr_{N}(BB^*)} \leq \norm{A} \norm{B}.
	\end{equation}
	Hence, thanks to Equation \eqref{assump}, one can bound the terms in~\eqref{guewigner3},~\eqref{guewigner4}, \eqref{guewigner6} and \eqref{guewigner7} by 
	\begin{equation}
		\label{extraterm3}
		\frac{y^2\norm{P}^2}{N} \E\left[|M_N|^{2n-2}  \norm{B_1(Z^N)}\dots \norm{B_l(Z^N)}\right], 
	\end{equation}
	for some polynomials $B_1$ to $B_l$ and $Z^N=(X^N,Y^N,A^N)$. Similarly one can bound the term in~\eqref{guewigner5}~by
	\begin{equation}
		\label{extraterm4}
		\frac{y^2\norm{P}}{\sqrt{N}} \E\left[|M_N|^{2n-1}  \norm{B_1(Z^N)}\dots \norm{B_l(Z^N)}\right]. 
	\end{equation}

    Then by using Lemma \ref{podvfldmvd} with $\alpha=y^2/N^{1/2-\varepsilon}$, $\mathcal{X}=N^{-\frac{g}{2n-g}\varepsilon} M_N$, $\mathcal{Y}=\norm{B_1(Z^N)}\dots \norm{B_l(Z^N)}$, and $g=1$ or $2$, we get that \eqref{extraterm3} and \eqref{extraterm4} can be upper bounded by
    $$ N^{-\varepsilon} \E\left[|M_N|^{2n}\right] + \left(\frac{y^2}{N^{1/2-\varepsilon}}\right)^{2n} \E\left[ \left(\norm{B_1(Z^N)}\dots \norm{B_l(Z^N)}\right)^{\frac{2n}{g}} \right]. $$
    Thus, thanks to Lemma \ref{docsocscddssdc} which we use with $\gamma=\varepsilon$, there exists a constant $C_n$ such that the terms in Equations \eqref{extraterm3} and \eqref{extraterm4} can be upper bounded by
    \begin{equation*}
        N^{-\varepsilon} \E\left[|M_N|^{2n}\right] + C_n \left(\frac{y^2}{N^{1/2-2\varepsilon}}\right)^{2n}.
    \end{equation*}
    Thus after summing over $i,j$, Equation \eqref{guewigner full set} yields
	\begin{align}
		\label{guewigner22}
		&\left| e^{t/2}\E\left[ \tr_N\left( \partial_s(Q(X_t^N)\widetilde{\#}P) Y_s^N \right) M_N^{n-1} \overline{M_N}^{n} \right] - \frac{1}{N}\E\left[ \tr_N^{\otimes 2}\left( \partial_s \left(\partial_sQ(X_t^N)\widetilde{\#}P\right)  \right) M_N^{n-1} \overline{M_N}^{n} \right] \right|  \nonumber \\
		&\leq \sum_{1\leq i,j\leq N} R^{N,n}_{i,j} + N^{-\varepsilon} \E\left[|M_N|^{2n}\right] + C_n \left(\frac{y^2}{N^{1/2-2\varepsilon}}\right)^{2n}.
	\end{align}

	\textbf{Step 3 (Bounding the second order error term):} Since we have an explicit expression of the remainder $R_{i,j}^{N,n}$ in Proposition \ref{cuulantmexpansion}, we define $ M_N^{u,i,j,s}$ just like $M_N$ but with the entries $(i,j)$ and $(j,i)$ of the matrix $Y_s^N$ multiplied by $u$. Then we have that~$R^{N,n}_{i,j}$ is bounded by a linear combination (whose coefficients are independent of $i,j$) of terms of the form,
	\begin{equation}
		\label{kjfd2}
		\frac{y^3}{N^{3/2-r/2}} \E\left[ \left| (Y_s^N)_{i,j} \right|^{r} \int_{0}^{1} \left|(A_1^{u,i,j,s})_{e_1,e_2} \dots (A_3^{u,i,j,s})_{e_5,e_6}\right|\ |M_N^{u,i,j,s}|^{2n-m} du\right] ,
	\end{equation}
	where $r,m\in [1,3 ]$, $e_1,\dots,e_6\in \{i,j\}$, and $A_1^{u,i,j,s}$ to $A_3^{u,i,j,s}$ are matrices which are evaluation of elements in $\F_{d,q}$. Besides, thanks to $P$, $m$ of those matrices have rank $1$. Furthermore the indices $e_1,\dots,e_6$ come from the derivatives with respect to $\Re((Y_s^N)_{i,j})$ and $\Im((Y_s^N)_{i,j})$ which makes $E_{i,j}$ and $E_{j,i}$ appear, coupled with the fact that $\tr_{N}(A E_{i,j}) = A_{j,i}$. Consequently, half of the indices $e_l$ are equal to $i$ and the other half to $j$. This fact will come in handy in the next few computations. 
	
	Next, let $R\in\F_{d,q}$ be such that $A_1^{u,i,j,s} = R(\widetilde{Z}^N)$ where $\widetilde{Z}^N$ is defined like $Z^N$ but with the entries $(j,i)$ and $(i,j)$ of the matrix $Y_s^N$ multiplied by $u\in[0,1]$. Then with $\ev(A\otimes B) = A(Z^N)\otimes B(\widetilde{Z}^N)$ and $\ev(A\otimes B\otimes C) = A(Z^N)\otimes B(Z^N) \otimes B(\widetilde{Z}^N)$,
    \begin{align*}
        R(Z^N) - R(\widetilde{Z}^N) &= (1-u)\ \ev\circ\partial_sR\# \left(E_{i,j}(Y_s^N)_{i,j}+E_{j,i}(Y_s^N)_{j,i}\right) \\
        &= (1-u)\ \partial_sR(Z^N)\# \left(E_{i,j}(Y_s^N)_{i,j}+E_{j,i}(Y_s^N)_{j,i}\right) \\
        &\quad - (1-u)^2\ \ev\circ(\id\otimes\partial_s)\circ\partial_sR(Z^N)\# \Big(\left(E_{i,j}(Y_s^N)_{i,j}+E_{j,i}(Y_s^N)_{j,i}\right), \\ 
        &\quad\quad\quad\quad\quad\quad\quad\quad\quad\quad\quad\quad\quad\quad\quad\quad\quad\quad \left(E_{i,j}(Y_s^N)_{i,j}+E_{j,i}(Y_s^N)_{j,i}\right) \Big).
    \end{align*}
    Consequently, since $|1-u|\leq 1$, one can find matrices $A_1^1$, $A^2_1$, $A^3_1$, polynomials $B_1,\dots,B_{l}, C_1,\dots,C_{l'},L$ such~that
	\begin{align*}
		\left|(A_1^{t,i,j,s})_{e_1,e_2}\right| \leq \left|(A_1^{1})_{e_1,e_2}\right| &+ |y (Y_s^N)_{i,j}| \left(\left|(A_1^{2})_{e_1,\widetilde{e_2}}\right| + \left|(A_1^{3})_{\widetilde{e_1},e_2}\right|\right) \norm{B_1(Z^N)}\dots \norm{B_l(Z^N)} \\
		&+ |y (Y_s^N)_{i,j}|^2 \norm{C_1(Z^N)}\dots \norm{C_{l'}(Z^N)} L\left( \left|(Y_s^N)_{i,j}\right|\right) \norm{P}^{h_1},
	\end{align*}
	where $\widetilde{e_1}$ and $\widetilde{e_2}$ are equal to either $i$ or $j$. Besides if $A_1^{1}$ is of rank $1$, then one can pick $A_1^{2}$ and $A_1^{3}$ of rank at most $1$, and $h_1=1$. In every other case $h_1=0$. One can also find polynomials $D_1,\dots,D_{l},L'$ such that 
	\begin{align*}
		\left|M_N^{t,i,j,s} - M_N\right| \leq |y (Y_s^N)_{i,j}| \norm{D_1(Z^N)}\dots \norm{D_l(Z^N)} L'\left( \left|(Y_s^N)_{i,j}\right|\right) \norm{P},
	\end{align*}
	which means that
	$$ |M_N^{t,i,j,s}|^{2n-m}  \leq \sum_{i=1}^{2n-m} \binom{2n-m}{p} |y (Y_s^N)_{i,j}|^p \norm{D_1(Z^N)}^p\dots \norm{D_l(Z^N)}^p L'\left( \left|(Y_s^N)_{i,j}\right|\right)^p \norm{P}^p |M_N|^{2n-m-p}.$$
	
	\noindent Thus, one can bound \eqref{kjfd2} by a linear combination (whose coefficients are independent of $i,j$) of terms of the form 
	\begin{align}
		\label{kjfd3}
		\frac{y^{3+p}\norm{P}^p}{N^{(3+p)/2}} \E\Bigg[ &\left(\left|(A_1^{1})_{e_1,e_2}\right| + \frac{y}{\sqrt{N}} \left(\left|(A_1^{2})_{e_1,\widetilde{e_2}}\right| + \left|(A_1^{3})_{\widetilde{e_1},e_2}\right|\right) + \frac{y^2\norm{P}^{h_1}}{N}\right) \\
		&\times \left(\left|(A_2^{1})_{e_3,e_4}\right| + \frac{y}{\sqrt{N}} \left(\left|(A_2^{2})_{e_3,\widetilde{e_4}}\right| + \left|(A_2^{3})_{\widetilde{e_3},e_4}\right|\right) + \frac{y^2\norm{P}^{h_2}}{N}\right) \nonumber \\
		&\times \left(\left|(A_3^{1})_{e_5,e_6}\right| + \frac{y}{\sqrt{N}} \left(\left|(A_3^{2})_{e_5,\widetilde{e_6}}\right| + \left|(A_3^{3})_{\widetilde{e_5},e_6}\right|\right) + \frac{y^2\norm{P}^{h_3}}{N}\right) \nonumber \\
        &\quad\quad\times \norm{H_1(Z^N)}\dots \norm{H_l(Z^N)} L\left( \sqrt{N}\left|(Y_s^N)_{i,j}\right|\right) |M_N|^{2n-m-p} \Bigg], \nonumber
	\end{align}
	
	\noindent for some polynomials $H_1,\dots,H_l,L$. Since the non-normalized trace of a rank $1$ matrix is smaller than its norm, given $A,B,C$ matrices of size $N$, assuming that $g$ of them have rank $1$, one has that
	\begin{align}
	\label{order11}
		\sum_{1\leq i,j\leq N} |A_{i,j} B_{i,j} C_{i,j}| &\leq \sqrt{\sum_{1\leq i,j\leq N} |A_{i,j} B_{i,j}|^2} \sqrt{\sum_{1\leq i,j\leq N} |C_{i,j}|^2}  \\
		&\leq \sqrt{\tr_N(C^*C) } \times\max\{\norm{B} \sqrt{\tr_N(A^*A)}, \norm{A} \sqrt{\tr_N(B^*B)}\} \nonumber \\
		&\leq N^{(1-g/2)\wedge 0} \norm{A} \norm{B} \norm{C}, \nonumber
	\end{align}
	\noindent and
	\begin{align}
	\nonumber
		\sum_{1\leq i,j\leq N} |A_{i,i} B_{j,j} C_{i,j}| &\leq \sqrt{\sum_{1\leq i,j\leq N} |A_{i,i} B_{j,j}|^2} \sqrt{\sum_{1\leq i,j\leq N} |C_{i,j}|^2}  \\
		&\leq \sqrt{\tr_N(A^*A) \tr_N(B^*B) \tr_N(C^*C) } \label{order12} \\
		&\leq N^{(3-g)/2} \norm{A} \norm{B} \norm{C}. \nonumber
	\end{align}	

	\noindent Similarly, we have that
	\begin{align}
		\label{order2}
		&\sum_{1\leq i,j\leq N} |A_{i,j} B_{i,j} C_{i,i}| \leq N^{(3-g)/2}  \norm{A} \norm{B} \norm{C},  \\
		&\sum_{1\leq i,j\leq N} |A_{j,j} B_{i,i} C_{i,i}| \leq N^{2-g/2}  \norm{A} \norm{B} \norm{C}, \nonumber \\
		&\sum_{1\leq i,j\leq N} |A_{j,i} B_{i,i} C_{i,i}| \leq N^{2-g/2}  \norm{A} \norm{B} \norm{C}, \nonumber \\
		&\sum_{1\leq i,j\leq N} |A_{i,i} B_{i,i} C_{i,i}| \leq N^{2-(g\wedge 2)/2}  \norm{A} \norm{B} \norm{C}. \nonumber
	\end{align}

	\noindent If we only have two matrices, $g$ of them being of rank $1$, 
	\begin{align}
		\label{order3}
		&\sum_{1\leq i,j\leq N} |A_{i,j} B_{i,j}| \leq N^{1-g/2}  \norm{A} \norm{B},  \\
		&\sum_{1\leq i,j\leq N} |A_{i,j} B_{i,i}| \leq N^{(3-g)/2}  \norm{A} \norm{B}, \nonumber \\
		&\sum_{1\leq i,j\leq N} |A_{i,i} B_{i,i}| \leq N^{2-g/2}  \norm{A} \norm{B}. \nonumber
	\end{align}

	\noindent Further, with one matrix, with $g=1$ if it is of rank $1$ and $0$ otherwise,
	\begin{align}
		\label{order4}
		&\sum_{1\leq i,j\leq N} |A_{i,j}| \leq N^{(3-g)/2}  \norm{A},  \\
		&\sum_{1\leq i,j\leq N} |A_{i,i}| \leq N^{2-g/2}  \norm{A}. \nonumber
	\end{align}
	
	\noindent Thus thanks to Equations \eqref{order11} and \eqref{order12}, we have that 
	$$ \frac{y^{3}}{N^{3/2}} \sum_{i,j} \left|(A_1^{1})_{e_1,e_2} (A_2^{1})_{e_3,e_4} (A_3^{1})_{e_5,e_6} \right| \leq \frac{y^3}{N^{m/2}} \norm{A_1^1}\norm{A_2^1}\norm{A_3^1}. $$
	Thanks to the first two line of Equation \eqref{order2}, we also get that
	$$ \frac{y^{4}}{N^{2}} \sum_{i,j} \left(\left|(A_1^{2})_{e_1,\widetilde{e_2}}\right| + \left|(A_1^{3})_{\widetilde{e_1},e_2}\right|\right) \left| (A_2^{1})_{e_3,e_4} (A_3^{1})_{e_5,e_6} \right| \leq \frac{y^4}{N^{m/2}}\left(\norm{A_1^2}+\norm{A_1^3}\right)\norm{A_2^1}\norm{A_3^1}. $$
 
	\noindent With all of Equation \eqref{order2}, since we assumed that $y^4\leq \sqrt{N}$, we have that
    \begin{align*}
        &\frac{y^{5}}{N^{5/2}} \sum_{i,j} \left(\left|(A_1^{2})_{e_1,\widetilde{e_2}}\right| + \left|(A_1^{3})_{\widetilde{e_1},e_2}\right|\right)\left(\left|(A_2^{2})_{e_3,\widetilde{e_4}}\right| + \left|(A_2^{3})_{\widetilde{e_3},e_4}\right|\right) \left|(A_3^{1})_{e_5,e_6} \right| \\
        &\leq \frac{y^{5}}{N^{(1+m)/2}} \left(\norm{A_1^2}+\norm{A_1^3}\right)\left(\norm{A_2^2}+\norm{A_2^3}\right)\norm{A_3^1} \\
        &\leq \frac{y}{N^{m/2}} \left(\norm{A_1^2}+\norm{A_1^3}\right)\left(\norm{A_2^2}+\norm{A_2^3}\right)\norm{A_3^1},
    \end{align*}
    and
    \begin{align*}
        &\frac{y^{6}}{N^{3}} \sum_{i,j} \left(\left|(A_1^{2})_{e_1,\widetilde{e_2}}\right| + \left|(A_1^{3})_{\widetilde{e_1},e_2}\right|\right) \left(\left|(A_2^{2})_{e_3,\widetilde{e_4}}\right| + \left|(A_2^{3})_{\widetilde{e_3},e_4}\right|\right) \left(\left|(A_3^{2})_{e_5,\widetilde{e_6}}\right| + \left|(A_3^{3})_{\widetilde{e_5},e_6}\right|\right) \\
        &\leq \frac{y^{6}}{N^{1 + (m\wedge 2)/2}} \left(\norm{A_1^2}+\norm{A_1^3}\right)\left(\norm{A_2^2}+\norm{A_2^3}\right) \left(\norm{A_3^2}+\norm{A_3^3}\right) \\
        &\leq \frac{y^2}{N^{m/2}} \left(\norm{A_1^2}+\norm{A_1^3}\right) \left(\norm{A_2^2}+\norm{A_2^3}\right) \left(\norm{A_3^2}+\norm{A_3^3}\right).
    \end{align*}
    
	\noindent Then thanks to Equation \eqref{order3}, we also have
    \begin{align*}
        &\frac{y^{5}}{N^{5/2}} \sum_{i,j} \left| (A_2^{1})_{e_3,e_4} (A_3^{1})_{e_5,e_6} \right|
        \leq \frac{y^5}{N^{(2+m\wedge 2)/2}} \norm{A_2^1}\norm{A_3^1}
        \leq \frac{y}{N^{m/2}} \norm{A_2^1}\norm{A_3^1},
    \end{align*}
    and
    \begin{align*}
        \frac{y^{6}}{N^{3}} \sum_{i,j} \left(\left|(A_2^{2})_{e_3,\widetilde{e_4}}\right| + \left|(A_2^{3})_{\widetilde{e_3},e_4}\right|\right) \left| (A_3^{1})_{e_5,e_6} \right| &\leq \frac{y^6}{N^{1+(m\wedge 2)/2}} \left(\norm{A_2^2}+\norm{A_2^3}\right)\norm{A_3^1} \\
        &\leq \frac{y^2}{N^{m/2}} \left(\norm{A_2^2}+\norm{A_2^3}\right)\norm{A_3^1},
    \end{align*}
    as well as
    \begin{align*}
        &\frac{y^{7}}{N^{7/2}} \sum_{i,j} \left(\left|(A_2^{2})_{e_3,\widetilde{e_4}}\right| + \left|(A_2^{3})_{\widetilde{e_3},e_4}\right|\right) \left(\left|(A_3^{2})_{e_5,\widetilde{e_6}}\right| + \left|(A_3^{3})_{\widetilde{e_5},e_6}\right|\right) \\
        &\qquad\leq \frac{y^7}{N^{(3+m\wedge 2)/2}} \left(\norm{A_2^2}+\norm{A_2^3}\right) \left(\norm{A_3^2}+\norm{A_3^3}\right) \\
        &\qquad\leq \frac{y^3}{N^{(1+m)/2}} \left(\norm{A_2^2}+\norm{A_2^3}\right) \left(\norm{A_3^2}+\norm{A_3^3}\right).
    \end{align*}
	
	\noindent Consequently, thanks to Lemmas~\ref{podvfldmvd} (with $g=m+p$) and \ref{docsocscddssdc} as well as the moment assumption in~\eqref{assump}, we get that after summing over $i$ and $j$, there exists a constant $C_{n}$ such that the term in Equation \eqref{kjfd3} is bounded by 
	\begin{align}
		N^{-\varepsilon} \E\left[ |M_N|^{2n} \right] + C_{n} \left(\frac{y^4\norm{P}}{N^{1/2-2\varepsilon}}\right)^{2n}. \nonumber
	\end{align}
 
	Thus from Equations \eqref{guewigner22} and \eqref{kjfd2} combined with the above estimate, we get that for a constant~$C_n$,
	\begin{align*}
		&\left| e^{t/2}\E\left[ \tr_N\left( \partial_s(Q(X_t^N)\widetilde{\#}P) Y_s^N \right) M_N^{n-1} \overline{M_N}^{n} \right] - \frac{1}{N}\E\left[ \tr_N^{\otimes 2}\left( \partial_s \left(\partial_sQ(X_t^N)\widetilde{\#}P\right)  \right) M_N^{n-1} \overline{M_N}^{n} \right] \right|  \nonumber \\
		&\leq C_{n}\left( N^{-\varepsilon} \E\left[ |M_N|^{2n} \right] + \left(\frac{y^4\norm{P}}{N^{1/2-2\varepsilon}}\right)^{2n} \right).
	\end{align*}
	Hence by plugging this estimate back into Equation~\eqref{sdpdso2}, for $N$ sufficiently large,
    \begin{align*}
        \E\left[ |M_N|^{2n} \right] \leq \frac{C_n}{1-\frac{d}{2}C_n N^{-\varepsilon}} \left(\frac{y^4}{N^{1-2\varepsilon}}\right)^{2n} \leq 2C_n \left(\frac{y^4\norm{P}}{N^{1-2\varepsilon}}\right)^{2n}.
	\end{align*}
	Consequently from Markov's inequality we have, for any $\varepsilon>0$, with high probability that
	$$ \left\langle \x, Q(Y^N,A^N) \y \right\rangle = \E\left[ \left\langle \x, Q(X^N,A^N) \y \right\rangle \right] + \mathcal{O}\left(\frac{y^{4}\norm{P}}{N^{1/2-\varepsilon}}\right).$$
	Since $\norm{P}=\norm{\x}_2\norm{\y}_2$, this proves~\eqref{scalar concentration estimate}.\\
	
	\textbf{Step 4 (The case of the GUE and the GOE):} In the case where $Y_s^N$ is a GOE random matrix, then by using the Schwinger-Dyson equations, i.e. Proposition \ref{3SD} (with a slight modification to take into account the term $ M_N^{n-1} \overline{M_N}^{n}$), Equation \eqref{guewigner full set} simplifies into
	\begin{align*}
		&e^{t/2}\E\left[ \tr_N\left( \partial_s(Q(X_t^N)\widetilde{\#}P)\ Y_s^N \right) M_N^{n-1} \overline{M_N}^{n} \right] \nonumber\\
		&= \frac{1}{N}\E\left[ \tr_N^{\otimes 2}\left( \partial_s \left(\partial_sQ(X_t^N)\widetilde{\#}P\right) \right) M_N^{n-1} \overline{M_N}^{n} \right] \nonumber\\
		&\ \quad+ \frac{1}{N} \left(\frac{n}{2}-1\right) \E\left[ \tr_N\left( \left(\partial_sQ(X_t^N)\widetilde{\#}P\right) \left(\partial_sQ(X_0^N)\widetilde{\#}P \right)\right) M_N^{n-2} \overline{M_N}^{n} \right] \nonumber\\
		&\ \quad+ \frac{n}{2N} \E\left[ \tr_N\left( \left(\partial_sQ(X_t^N)\widetilde{\#}P \right) \left( \partial_sQ(X_0^N)\widetilde{\#}P\right)^* \right) M_N^{n-1} \overline{M_N}^{n-1} \right] \\
		&\ \quad+\frac{1}{N}\E\left[ \tr_N\left( h\left(\partial_s(Q(X_t^N)\widetilde{\#}P)\right) \right) M_N^{n-1} \overline{M_N}^{n} \right] \nonumber\\
		&\ \quad+ \frac{1}{N} \left(\frac{n}{2}-1\right) \E\left[ \tr_N\left(\left( \partial_s(Q(X_t^N)\widetilde{\#}P) \right)^T \left(\partial_s(Q(X_0^N)\widetilde{\#}P)\right) \right) M_N^{n-2} \overline{M_N}^{n} \right] \nonumber\\
		&\ \quad+ \frac{n}{2N} \E\left[ \tr_N\left( \left(\partial_s(Q(X_t^N)\widetilde{\#}P)\right)^T \left(\partial_s(Q(X_0^N)\widetilde{\#}P)\right)^* \right) M_N^{n-1} \overline{M_N}^{n-1} \right], \nonumber
	\end{align*}
	where $A^T$ is the transpose of $A$, and $h$ is the linear map such that $h(A\otimes B) = A^TB$. Thus, we get that for some constant $C_n$,
	\begin{align*}
		&\left| e^{t/2}\E\left[ \tr_N\left( \partial_s(Q(X_t^N)\widetilde{\#}P) Y_s^N \right) M_N^{n-1} \overline{M_N}^{n} \right] - \frac{1}{N}\E\left[ \tr_N^{\otimes 2}\left( \partial_s \left(\partial_sQ(X_t^N)\widetilde{\#}P\right)  \right) M_N^{n-1} \overline{M_N}^{n} \right] \right|  \nonumber \\
		&\leq C_{n}\left( N^{-\varepsilon} \E\left[ |M_N|^{2n} \right] + \left(\frac{y^2\norm{P}}{N^{1/2-2\varepsilon}}\right)^{2n} \right).
	\end{align*}
	Hence, by using Equation \eqref{sdpdso2}, one concludes that $\E[|M_N|^{2n}] = \mathcal{O}(y^{4n}N^{-n(1-\varepsilon)}\norm{P}^{2n})$ for any $n$. 
 
 The case where $Y_s^N$ is a GUE random matrix is even simpler. Indeed, we compute that
	\begin{align*}
		&e^{t/2}\E\left[ \tr_N\left( \partial_s(Q(X_t^N)\widetilde{\#}P)\ Y_s^N \right) M_N^{n-1} \overline{M_N}^{n} \right] \nonumber\\
		&= \frac{1}{N}\E\left[ \tr_N^{\otimes 2}\left( \partial_s \left(\partial_sQ(X_t^N)\widetilde{\#}P\right) \right) M_N^{n-1} \overline{M_N}^{n} \right] \nonumber\\
		&\ \quad+ \frac{1}{N} \left(\frac{n}{2}-1\right) \E\left[ \tr_N\left( \left(\partial_sQ(X_t^N)\widetilde{\#}P\right) \left(\partial_sQ(X_0^N)\widetilde{\#}P \right)\right) M_N^{n-2} \overline{M_N}^{n} \right] \nonumber\\
		&\ \quad+ \frac{n}{2N} \E\left[ \tr_N\left( \left(\partial_sQ(X_t^N)\widetilde{\#}P \right) \left( \partial_sQ(X_0^N)\widetilde{\#}P\right)^* \right) M_N^{n-1} \overline{M_N}^{n-1} \right],
	\end{align*}
	and the result follows.
\end{proof}

\section{Proof of the main theorems}
\label{sec:proof main theorems}

\begin{proof}[Proof of Theorem \ref{main_result}]
	To prove Equation \eqref{trace}, we combine Proposition \ref{concentrtrace} with Lemma 3.6 of \cite{topoexp}. Indeed, with $Q$ defined as in Proposition \ref{concentrtrace}, we have for any $\varepsilon>0$, that with high probability
	\begin{equation}
		\label{oifesoids}
		\ts_{N}\left( Q(Y^N,A^N) \right) = \E\left[ \ts_{N}\left(Q(X^N,A^N) \right) \right] + \mathcal O\left( N^{\varepsilon} \frac{y^4}{N} \right) .
	\end{equation}
	Whereas if we use Lemma 3.6 of \cite{topoexp} (with $n=0$), where  $x=(x_1,\dots,x_d)$ are free semicircular variables, free from $A^N\in\M_N(\C)^q$, we get that
	\begin{align}
		\label{oires}
		\E\left[ \ts_{N}\left(Q(X^N,A^N) \right) \right] = \ &\tau_N\left( Q(x,A^N) \right) \\
		&+ \frac{1}{N^2} \int_{0}^{\infty} \int_{0}^{t} \bigg(\int_{[0,1]^4} e^{-r-t} \E\left[\tau_{N}(R_{\alpha,\beta,\delta,\gamma,r,t} )\right]\  d\alpha d\beta d\gamma d\delta\bigg) dr dt, \nonumber
	\end{align}
	where $R_{\alpha,\beta,\delta,\gamma,r,t}$ is such that for some constant $C$ independent of $\alpha,\beta,\delta,\gamma,r,t,y$ and $N$,
	$$\E\left[\norm{R_{\alpha,\beta,\delta,\gamma,r,t}}\right] \leq C y^4.$$
	Hence we have that
	\begin{equation}
	\label{pfdsspd}
		\E\left[ \ts_{N}\left(Q(X^N,A^N) \right) \right] = \tau_N\left( Q(x,A^N) \right) + \mathcal{O}\left( \frac{y^4}{N^2} \right). 
	\end{equation}
	Finally, with $f_1$ to $f_k$ defined as in Theorem \ref{main_result}, we set 
	$$
	\widetilde{f}_i: t \mapsto \left\{
	\begin{array}{ll}
		e^{\i y_i t} & \mbox{if } f_i\neq \id_{\R}, \\
		t & \mbox{else.}
	\end{array}
	\right.
	$$
	Then with $Q=\widetilde{f}_1\left(P_1\right)\dots \widetilde{f}_k\left(P_k\right)$, and $\mu_i$ the Dirac measure in $1$ if $f_i=\id_{\R}$, thanks to the functional calculus we have
	\begin{equation}
		f_1(P_1(Y^N,A^N)) \cdots f_k(P_k(Y^N,A^N)) = \int_{\R^k} Q\left(Y^N,A^N\right)\ d\mu_1(y_1)\dots d\mu_k(y_k).
	\end{equation}
	Thus by combining Equations \eqref{oifesoids} and \eqref{pfdsspd}, we get Equation \eqref{trace} after integrating over $y$.	
	
	 Besides in the case where we are working only with GUE and GOE random matrices, then one can replace $y^4$ by $y^2$ in Equation \eqref{oifesoids}. Then if $y^2\leq N$, Equation \eqref{oires} let us conclude, whereas if $y^2\geq N$, then we simply use the fact that
	$$ \ts_{N}\left( Q(Y^N,A^N) \right) - \tau_N\left( Q(x,A^N) \right) = \mathcal{O}(1) = \mathcal{O}\left( \frac{y^2}{N} \right). $$

	To prove Equation \eqref{scalar}, we also start by using Proposition \ref{concentrcoeff} to show that for any $\varepsilon>0$, with high probability
	\begin{align}
		&\left\langle \x, Q(Y^N,A^N) \y \right\rangle = \E\left[ \left\langle \x, Q(X^N,A^N) \y \right\rangle \right] + \mathcal O\left( N^{\varepsilon} \frac{y^4}{\sqrt{N}} \norm{\x}_2\norm{\y}_2\right),
	\end{align}
	where one can replace $y^4$ by $y^2$ in the case where we are working only with GUE and GOE random matrices. Then we need to estimate $\E\left[ \left\langle \x, Q(X^N,A^N) \y \right\rangle \right]$. Thanks again to Lemma 3.6 of \cite{topoexp}, we get that 
	\begin{align*}
		&\E\left[ \left\langle \x, Q(X^N,A^N) \y \right\rangle \right] \\
		&= \E\left[ \tr_N\left( Q(X^N,A^N) \y \x^* \right) \right] \\
		&= N\tau_N\left( Q(x,A^N) \y \x^* \right) + \frac{1}{N} \int_{0}^{\infty} \int_{0}^{t} \left(\int_{[0,1]^4} e^{-r-t} \E\left[\tau_{N}(R_{\alpha,\beta,\delta,\gamma,r,t} \y\x^*)\right]\  d\alpha d\beta d\gamma d\delta\right) dr dt \\
		&= \tr_N\left( E_{\M_N(\C)}\left[Q(x,A^N)\right] \y \x^* \right) \\
		&\quad\quad\quad\quad\quad + \frac{1}{N^2} \int_{0}^{\infty} \int_{0}^{t} \left(\int_{[0,1]^4} e^{-r-t} \E\left[\tr_N\left( E_{\M_N(\C)}\left[R_{\alpha,\beta,\delta,\gamma,r,t}\right] \y\x^*\right)\right]\  d\alpha d\beta d\gamma d\delta\right) dr dt \\
		&= \left\langle\x, E_{\M_N(\C)}\left[Q(x,A^N)\right] \y \right\rangle \\
		&\quad\quad\quad\quad\quad + \frac{1}{N^2} \int_{0}^{\infty} \int_{0}^{t}\left( \int_{[0,1]^4} e^{-r-t} \E\left[\tr_N\left( E_{\M_N(\C)}\left[R_{\alpha,\beta,\delta,\gamma,r,t}\right] \y\x^*\right)\right]\  d\alpha d\beta d\gamma d\delta\right) dr dt,
	\end{align*}
	where $R_{\alpha,\beta,\delta,\gamma,r,t}$ is once again such that for some constant $C$ independent of $\alpha,\beta,\delta,\gamma,r,t,y$ and $N$,
	$$\E\left[\norm{R_{\alpha,\beta,\delta,\gamma,r,t}}\right] \leq C y^4.$$
	Consequently since $P=\y\x^*$ is of rank $1$ and $\norm{P}=\norm{\x}_2\norm{\y}_2$,
	$$ \E\left[ \left\langle \x, Q(X^N,A^N) \y \right\rangle \right] =  \left\langle\x, E_{\M_N(\C)}\left[Q(x,A^N)\right] \y \right\rangle + \mathcal{O}\left(\frac{y^4}{N^2}\norm{\x}_2\norm{\y}_2\right).$$
	The rest of the proof follows as for Equation \eqref{trace} with the difference that after integrating over $y$, we use the fact that $b\mapsto E_{\M_N(\C)}[b]$ is a continuous linear functional to switch the integral and the conditional expectation.
\end{proof}

\begin{proof}[Proof of Theorem \ref{dliuf}]
	Let $Q_1$ to $Q_p$ be non commutative polynomials, $i_1,\dots,i_p\in[1,k]$ be such that for every $j$, $\tau(Q_j(a_{i_j}))=0$ and  if $j<p$, $i_j\neq i_{j+1}$. Then if we set $u_{2j-1}^N:= e^{\i (y_{i_j}^N -y_{i_{j-1}}^N) P(x)}$ (with the convention $i_0 = i_p$) and $u_{2j}^N = Q_j(A_{i_j}^N)$, one can apply Theorem \ref{main_result} to obtain that with high probability
	$$ \tau_N\left( Q_1(a_{i_1}^N) \dots Q_p(a_{i_p}^N)\right) = \tau_{N}\left( u_1^N u_2^N \dots u_{2p-1}^N u_{2p}^N \right) + \mathcal{O}\left( \frac{1}{N}{\max\limits_{i\neq j} | y_i^N-y_j^N |^4} \right). $$
	Further, if $Y^N$ is a family of GUE and GOE matrices, then
	\begin{align*}
	\tau_N\left( Q_1(a_{i_1}^N) \dots Q_p(a_{i_p}^N)\right) = \tau_{N}\left( u_1^N u_2^N \dots u_{2p-1}^N u_{2p}^N \right) + \mathcal{O}\left(\frac{1}{N} \max\limits_{i\neq j} | y_i^N-y_j^N |^2 \right). 
	\end{align*}
	Since by assumption for any $i,j$, $|y_i^N-y_j^N|\ll N^{1/2}$ in the case of GUE and GOE matrices, and $|y_i^N-y_j^N|\ll N^{1/4}$ in all generality, we get in both cases that with high probability
	$$ \tau_N\left( Q_1(a_{i_1}^N) \dots Q_p(a_{i_p}^N)\right) = \tau_{N}\left( u_1^N u_2^N \dots u_{2p-1}^N u_{2p}^N \right) + o(1). $$
	Thus the former equality is true almost surely by Borel-Cantelli.
	
	But then thanks to Proposition 11.4 of \cite{nica_speicher}, with $\mathrm{NC}[2p]$ the set of non-crossing partition of $[1,2p]$ (i.e. the set of partitions of $[1,2p]$ such that one cannot find $a,c$ in one block as well as $b,d$ in another block such that $a< b<c<d$), we can write 
	$$	\tau_{N}\left( u_1^N \dots u_{2p}^N \right) = \sum_{\pi\in \mathrm{NC}[2p]} \kappa_{\pi}(u_1^N, u_2^N ,\dots, u_{2p-1}^N, u_{2p}^N) ,$$
	where $\kappa_{\pi}(u_1^N, u_2^N ,\dots, u_{2p-1}^N, u_{2p}^N)$ are the so-called free cumulants. We do not need to define those in details in this proof, however it is important to note that $\kappa_{\pi}(u_1^N, u_2^N ,\dots, u_{2p-1}^N, u_{2p}^N)$ is a linear combination of products of traces such that $ u_{i}^N$ and $ u_{j}^N$ can only be in the same trace if $i$ and $j$ belong to the same block in the partition $\pi$. 
 
    Besides thanks to Theorem 11.16 of \cite{nica_speicher}, if $i,j$ belong to the same block in the partition $\pi$ and $ u_{i}^N$ is free from $ u_{j}^N$, then $ \kappa_{\pi}(u_1^N, u_2^N ,\dots, u_{2p-1}^N, u_{2p}^N)  =0$. Consequently the elements~$u_{j}^N$ for odd $j$ cannot be in the same trace as those for even $j$. More precisely, for $ \kappa_{\pi}(u_1^N, u_2^N ,\dots, u_{2p-1}^N, u_{2p}^N)$ not to be $0$, odd and even integers in $[1,2p]$ cannot be in the same block. Hence either $\{2\}$ is a singleton, or else $2$ belongs to $B$ a block of the partition $\pi$ which contains only even integers. Let $b>2$ be one of those integers, then since $\pi$ is non-crossing, one can iterate this process with $3$ instead of $2$ and $[3,b-1]$ instead of $[2,2p]$. Thus by a quick induction, one can see that necessarily for $\kappa_{\pi}(u_1^N, u_2^N ,\dots, u_{2p-1}^N, u_{2p}^N)$ to be non zero, there has to be a block in the partition $\pi$ which consists of a single element which is not $1$ (this is important since we did not assume that $i_1\neq i_p$, and hence one could have $u_1^N=1$). Then since for any~$j$,
	$$\lim_{N\to\infty}\tau_N(u_{2j}^N) = \tau(Q_j(a_{i_j}))=0, $$
	
	\noindent we only need to show that for any $j$, $\lim\limits_{N\to\infty} \tau\left(e^{\i (y_{i_j}^N -y_{i_{j-1}}^N) P(x)}\right) = 0.$ But since by assumption $i_j\neq i_{j-1}$, we have that $|y_{i_j}^N -y_{i_{j-1}}^N| \gg 1$. Consequently, all we need to prove is that $\lim\limits_{|t|\to\infty} \tau\left(e^{\i t P(x)}\right) = 0$. However thanks to \cite{density}, Theorem 1.1, (6), there exists  $f\in L^1(\R)$ such that
	$$ \tau\big(e^{\i t P(x)}\big) = \int_{\R} e^{\i t a} f(a)\ da.$$
	Thus, thanks to the Riemann-Lebesgue lemma, we have that 
		$$\lim_{N\to\infty}\tau_N(u_{2j-1}^N) =0. $$
	Hence we proved that for any $i_1,\dots,i_p\in[1,k]$ such that for $j<p$, $i_j\neq i_{j+1}$ and non commutative polynomials $Q_j$, such that for every $j$, $\tau(Q_j(a_{i_j}))=0$,  
	$$ \lim\limits_{N\to\infty} \tau_N\left( Q_1(a_{i_1}^N) \dots Q_p(a_{i_p}^N)\right) =0.$$
	Hence $(a_1^N,\dots,a_k^N)$ converges in distribution towards the free family $(a_1,\dots,a_k)$.
\end{proof}

\subsection*{Acknowledgments}

We thanks L\'asl\'o Erd\H{o}s for comments on the first draft. F.P.\ is supported by the Knut and Alice Wallenberg Foundation. K.S.\ is supported in parts by the Swedish Research Council (VR-2017-05195,VR-2021-04703), and the Knut and Alice Wallenberg Foundation. This material is based upon work supported by the National Science Foundation under Grant No. DMS-1928930 while K.S participated in a program hosted by the Mathematical Sciences Research Institute in Berkeley, CA, during the Fall 2021 semester.

\bibliographystyle{abbrv}

\begin{thebibliography}{10}

\bibitem{anderste}
Anderson, G. W. (2013). Convergence of the largest singular value of a polynomial in independent Wigner matrices. \emph{The Annals of Probability}, 41(3B), 2103-2181.

\bibitem{AndersonAnti}
Anderson, G. W. (2015). A local limit law for the empirical spectral distribution of the anticommutator of independent Wigner matrices. \emph{Annales de l'IHP Probabilit\'es et statistiques}, 51(3), 809-841.

\bibitem{alice}
Anderson, G. W., Guionnet, A., \& Zeitouni, O. (2010). \emph{An introduction to random matrices} (No. 118). Cambridge university press.

\bibitem{capbelin}
Belinschi, S. T., \& Capitaine, M. (2017). Spectral properties of polynomials in independent Wigner and deterministic matrices. \emph{Journal of Functional Analysis}, 273(12), 3901-3963.

\bibitem{calculdensalgo}
Belinschi, S. T., Mai, T., \& Speicher, R. (2017). Analytic subordination theory of operator-valued free additive convolution and the solution of a general random matrix problem. \emph{Journal f\"ur die reine und angewandte Mathematik} (Crelles Journal), 2017(732), 21-53.

\bibitem{benanti}
Benaych-Georges, F., \& Knowles, A. (2016). Lectures on the local semicircle law for Wigner matrices. \emph{arXiv preprint arXiv:1601.04055}.

\bibitem{Alex+Erdos+Knowles+Yau+Yin} 
Bloemendal, A., Erd\H{o}s, L., Knowles, A., Yau, H. T., \& Yin, J. (2014). Isotropic local laws for sample covariance and generalized Wigner matrices. \emph{Electronic Journal of Probability}, 19, 1-53.

\bibitem{BoutetdeMonvel} 
Boutet, D. M. A., \& Khorunzhy, A. (1999). Asymptotic distribution of smoothed eigenvalue density. II. Wigner random matrices. \emph{Random Operators and Stochastic Equations}, 7(2), 149-168.

\bibitem{thermalidsl}
Cipolloni, G., Erd\H{o}s, L., \& Schr\"oder, D. (2022). Thermalisation for Wigner matrices. \emph{Journal of Functional Analysis}, 282(8), 109394.

\bibitem{traceless}
Cipolloni, G., Erd\H{o}s, L., \& Schr\"oder, D. (2021). Optimal multi-resolvent local laws for Wigner matrices. \emph{arXiv preprint arXiv:2112.13693}.

\bibitem{ncfu}
Collins, B., Mai, T., Miyagawa, A., Parraud, F., \& Yin, S. (2021). Convergence for noncommutative rational functions evaluated in random matrices. \emph{arXiv preprint arXiv:2103.05962}.

\bibitem{un}
Collins, B., Guionnet, A., \& Parraud, F. (2022). On the operator norm of non-commutative polynomials in deterministic matrices and iid GUE matrices. \emph{Cambridge Journal of Mathematics}, 10(1), 195--260.  

\bibitem{rageop}
Cycon, H. L., Froese, R. G., Kirsch, W., \& Simon, B. (2009). \emph{Schr\"odinger operators: With application to quantum mechanics and global geometry}. Springer.

\bibitem{wigpague}
Dykema, K. (1993). On certain free product factors via an extended matrix model. \emph{Journal of functional analysis}, 112(1), 31-60.

\bibitem{ErdoesKnowlesYauYin} 
Erd\H{o}s, L., Knowles, A., Yau, H. T., \& Yin, J. (2013). The local semicircle law for a general class of random matrices. Electronic Journal of Probability, 18, 1-58.
	
\bibitem{lcp}
Erd\H{o}s, L., Kr\"uger, T., \& Nemish, Y. (2020). Local laws for polynomials of Wigner matrices. \emph{Journal of Functional Analysis}, 278(12), 108507.

\bibitem{ErdoesKruegerSchroeder} 
Erd\H{o}s, L., Kr\"uger, T., \& Schr\"oder, D. (2019). Random matrices with slow correlation decay. In \emph{Forum of Mathematics, Sigma} (Vol. 7). Cambridge University Press.
	
\bibitem{lclclc}
Erd\H{o}s, L., Schlein, B., \& Yau, H. T. (2009). Local semicircle law and complete delocalization for Wigner random matrices. \emph{Communications in Mathematical Physics}, 287(2), 641-655.

\bibitem{rigidity} 
Erd\H{o}s, L., Yau, H. T., \& Yin, J. (2012). Rigidity of eigenvalues of generalized Wigner matrices. \emph{Advances in Mathematics}, 229(3), 1435-1515.

\bibitem{nemish2}
Fronk, J., Krüger, T., \& Nemish, Y. (2023). Norm Convergence Rate for Multivariate Quadratic Polynomials of Wigner Matrices. \emph{arXiv preprint arXiv:2308.16778}.

\bibitem{aliceSDy}
Guionnet, A. (2019). \emph{Asymptotics of random matrices and related models: the uses of Dyson-Schwinger equations} (Vol. 130). American Mathematical Soc..
	
\bibitem{hT}
Haagerup, U., \& Thorbj\o rnsen, S. (2005). A new application of random matrices: Ext($\mathcal C^*_{red}(\mathbb F_2)$) is not a group. \emph{Annals of mathematics}, 162(2), 711-775.

\bibitem{moment} 
He, Y., \& Knowles, A. (2017). Mesoscopic eigenvalue statistics of Wigner matrices. \emph{The Annals of Applied Probability}, 27(3), 1510-1550.

\bibitem{He+Knowles}
He, Y., \& Knowles, A. (2021). Fluctuations of extreme eigenvalues of sparse Erd\H{o}s-R\'enyi graphs. \emph{Probability Theory and Related Fields}, 180(3), 985-1056.

\bibitem{HLY} 
Huang, J., Landon, B., \& Yau, H. T. (2020). Transition from Tracy-Widom to Gaussian fluctuations of extremal eigenvalues of sparse Erd\H{o}s-R\'enyi graphs. \emph{The Annals of Probability}, 48(2), 916-962.

\bibitem{ToddDavid}
Kemp, T., \& Zimmermann, D. (2020). Random matrices with log-range correlations, and log-Sobolev inequalities. In \emph{Annales Mathématiques Blaise Pascal} (Vol. 27, No. 2, pp. 207-232).

\bibitem{KKP} 
Khorunzhy, A. M., Khoruzhenko, B. A., \& Pastur, L. A. (1996). Asymptotic properties of large random matrices with independent entries. \emph{Journal of Mathematical Physics}, 37(10), 5033-5060.

\bibitem{KnowlesYinIsotropic} 
Knowles, A., \& Yin, J. (2013). The isotropic semicircle law and deformation of Wigner matrices. \emph{Communications on Pure and Applied Mathematics}, 66(11), 1663-1749.

\bibitem{KnowlesYinInterpolation}
Knowles, A., \& Yin, J. (2017). Anisotropic local laws for random matrices. \emph{Probability Theory and Related Fields}, 169(1), 257-352.

\bibitem{deformed} 
Lee, J. O., \& Schnelli, K. (2015). Edge universality for deformed Wigner matrices. \emph{Reviews in Mathematical Physics}, 27(08), 1550018.

\bibitem{LeeSchnellisparse}
Lee, J. O., \& Schnelli, K. (2018). Local law and Tracy-Widom limit for sparse random matrices. \emph{Probability Theory and Related Fields}, 171(1), 543-616.

\bibitem{male}
Male, C. (2012). The norm of polynomials in large random and deterministic matrices. \emph{Probability Theory and Related Fields}, 154(3), 477-532.

\bibitem{nemish1} Nemish, Y. (2017). Local law for the product of independent non-Hermitian random matrices with independent entries. \emph{Electronic Journal of Probability}, 22, 1-35.

\bibitem{nica_speicher}
Nica, A., \& Speicher, R. (2006). \emph{Lectures on the combinatorics of free probability} (Vol. 13). Cambridge University Press.

\bibitem{nikitopoulos}
Nikitopoulos, E. A. (2020). Noncommutative $ C^ k $ Functions and Fr\'{e}chet Derivatives of Operator Functions. \emph{arXiv preprint arXiv:2011.03126}.

\bibitem{lytova_pastur}
Lytova, A., \& Pastur, L. (2009). Central limit theorem for linear eigenvalue statistics of random matrices with independent entries. \emph{The Annals of Probability}, 37(5), 1778-1840.

\bibitem{topoexp}
Parraud, F. (2020). Asymptotic expansion of smooth functions in polynomials in deterministic matrices and iid GUE matrices. \emph{arXiv preprint arXiv:2011.04146}.

\bibitem{deux}
Parraud, F. (2022). On the operator norm of non-commutative polynomials in deterministic matrices and iid Haar unitary matrices. \emph{Probability Theory and Related Fields}, 182(3), 751--806.

\bibitem{density}
Shlyakhtenko, D., \& Skoufranis, P. (2015). Freely independent random variables with non-atomic distributions. \emph{Transactions of the American Mathematical Society}, 367(9), 6267-6291.
	
\bibitem{DVoi}
Voiculescu, D. (1991). Limit laws for random matrices and free products. \emph{Inventiones mathematicae}, 104(1), 201-220.

\bibitem{yinsldc}
Yin, S. (2018). Non-commutative rational functions in strong convergent random variables. \emph{Advances in Operator Theory}, 3(1), 178-192.























\end{thebibliography}

\end{document}